\documentclass[reqno]{amsart}

\title[Eigenvalues and entropy]{Eigenvalues and entropy of a Hitchin representation}
\author[R. Potrie]{Rafael Potrie}

\author[A. Sambarino]{Andr\'es Sambarino}

\thanks{The research leading to these results has received funding from 
the European Research Council under the {\em European Community}'s
 seventh Framework Programme (FP7/2007-2013)/ERC {\em grant agreement} ${\rm n^o}$ FP7-246918. R. P. was partially supported by FCE-2011-6749, CSIC grupo 618 and the Palis-Balzan project. A. S. was partially supported by CSIC group 618 and FCE-2011-6749.}

\usepackage{marvosym}
\usepackage{psfrag}
\usepackage[english]{babel}
\usepackage{mathrsfs}
\usepackage{verbatim}
\usepackage[ansinew]{inputenc}
\usepackage{amsmath,amsthm,amssymb,amscd}
\usepackage[all,cmtip]{xy}
\usepackage{xspace}
\usepackage[dvips]{color}
\usepackage{epsfig}
\usepackage{pb-diagram}
\usepackage{amsfonts}
\usepackage{graphicx}
\usepackage[mathcal]{eucal}
\usepackage{hyperref}

\usepackage{perpage}
\MakePerPage{footnote}

\newcommand{\R}{\mathbb{R}}

\renewcommand{\H}{\mathbb{H}}
\newcommand{\N}{\mathbb{N}}

\renewcommand{\P}{\mathbb{P}}

\renewcommand{\/}{\backslash}
\renewcommand{\k}{\kappa}

\newcommand{\Om}{\Omega}
\newcommand{\eps}{\varepsilon}
\newcommand{\tex}{\textrm}
\newcommand{\vacio}{\emptyset}

\newcommand{\G}{\sf{\Gamma}}
\renewcommand{\l}{p}
\newcommand{\<}{\left<}
\renewcommand{\>}{\right>}
\newcommand{\E}{\Sigma}

\newcommand{\g}{\gamma}
\newcommand{\z}{\zeta}
\newcommand{\w}{\widetilde}
\newcommand{\bord}{\partial}
\newcommand{\vo}[1]{\overline{#1}}

\newcommand{\cone}{\scr L}

\newcommand{\grupo}{\Delta}
\newcommand{\om}{\omega}
\newcommand{\posgen}{\scr F^{(2)}}
\renewcommand{\t}{\theta}

\newcommand{\UG}{\sf{U}\G}

\newcommand{\algo}[1]{c(#1)}
\renewcommand{\1}{\mathbf 1}
\renewcommand{\L}{\Lambda}

\newcommand{\conodual}{{\cone_\rho^\t}^*}
\newcommand{\D}{\cal{D}}

\newcommand{\scr}{\mathscr}
\renewcommand{\sf}[1]{{\mathsf{#1}}}
\renewcommand{\rm}{\mathbf}
\newcommand{\cal}{\mathcal}
\renewcommand{\frak}{\mathfrak}

\DeclareMathOperator{\ii}{i}
\DeclareMathOperator{\Lim}{L}

\DeclareMathOperator{\inte}{int}

\DeclareMathOperator{\SL}{SL}
\DeclareMathOperator{\PSL}{PSL}
\DeclareMathOperator{\Psp}{PSp}
\DeclareMathOperator{\GL}{GL}

\DeclareMathOperator{\grassman}{Gr}
\DeclareMathOperator{\clase}{C}

\DeclareMathOperator{\aire}{MinArea}

\DeclareMathOperator{\isom}{Isom}

\DeclareMathOperator{\Ge}{G_2}
\DeclareMathOperator{\PGL}{PGL}

\DeclareMathOperator{\PSO}{PSO}

\DeclareMathOperator{\holder}{Holder}
\DeclareMathOperator{\Livsic}{Livsic}
\DeclareMathOperator{\homa}{HA}
\DeclareMathOperator{\Hitchin}{Hitchin}

\DeclareMathOperator{\rk}{rank}

\DeclareMathOperator{\CAT}{CAT}

\DeclareMathOperator{\rest}{rt}

\newtheorem*{teo1}{Theorem A}
\newtheorem*{teoB}{Theorem B}
\newtheorem*{teoC}{Theorem C}
\newtheorem*{teoD}{Theorem D}

\newtheorem{teo}{Theorem}[section]
\newtheorem{cor}[teo]{Corollary}

\newtheorem{lema}[teo]{Lemma}

\newtheorem{prop}[teo]{Proposition}

\theoremstyle{definition}

\newtheorem{defi}[teo]{Definition}

\newtheorem{obs}[teo]{Remark}
\newtheorem{ex}[teo]{Example}

\theoremstyle{remark}

\hypersetup{
    colorlinks=true,       
    linkcolor=blue,          
urlcolor=blue,
citecolor=blue
}

\begin{document}

\begin{abstract} We show that the critical exponent of a representation $\rho$ in the Hitchin component of $\PSL(d,\R)$ is bounded above, the least upper bound being attained only in the Fuchsian locus. This provides a rigid inequality for the area of a minimal surface on $\rho\/X,$ where $X$ is the symmetric space of $\PSL(d,\R).$ The proof relies in a construction useful to prove a regularity statement: if the Frenet equivariant curve of $\rho$ is smooth, then $\rho$ is Fuchsian.
\end{abstract}

\maketitle

\tableofcontents


\section{Introduction}

Let $\E$ be a closed orientable surface of genus $\geq2.$ A representation $\pi_1\E\to\PSL(d,\R)$ is \emph{Fuchsian} if it factors as $$\pi_1\E\to\PSL(2,\R)\to\PSL(d,\R),$$ where the first arrow is a choice of a hyperbolic metric on $\E,$ and the second arrow is the (unique up to conjugation) irreducible linear action of $\SL(2,\R)$ on $\R^d.$\footnote[2]{This is standard, see Guichard \cite{guichard} for an explicit construction.}

A \emph{Hitchin component} of $\PSL(d,\R)$ is a connected
component of $$\frak
X(\pi_1\E,\PSL(d,\R))=\hom(\pi_1\E,\PSL(d,\R))/\PSL(d,\R)$$ that
contains a Fuchsian representation. Hitchin \cite{hitchin} proved
that there are either one, or two Hitchin components (according to
$d$ odd or even respectively), and that each of these components
is diffeomorphic to an open
$|\chi(\E)|\cdot\dim\PSL(d,\R)$-dimensional Euclidean ball. When
$d=2$ these two components correspond to the Teichm\"uller space
of $\E$ with a fixed orientation. A Hitchin component appears then
as a higher rank generalization of Teichm\"uller space. Denote by
$\Hitchin(\E,d)$ this (these) component(s).

The analogy with Teichm\"uller space is carried on. Labourie \cite{labourie} shows that a representation in $\Hitchin(\E,d)$ (from now on a \emph{Hitchin representation}) is discrete, irreducible and faithful, and consists of purely loxodromic elements. Guichard-Wienhard \cite{olivieranna} proved that Hitchin components are deformation spaces of geometric structures on closed manifolds. Bridgeman-Canary-Labourie-S. \cite{presion} provide a Weil-Petersson-type Riemannian metric on $\Hitchin(\E,d),$ invariant under the mapping class group of $\E.$

Denote by $X$ the symmetric space of $\PSL(d,\R),$ and by $d_X$ a distance on $X$ induced by a $\PSL(d,\R)$-invariant Riemannian
metric on $X.$ If $\grupo$ is a discrete subgroup of $\PSL(d,\R),$ the \emph{critical exponent} of $\grupo$ is defined by $$h_X(\grupo)=\lim_{s\to\infty}\frac{\log\#\{g\in\grupo:d_X(o,g\cdot o)\leq s\}}s,$$ for some (any) $o\in X.$

Introduced by Margulis \cite{margulistesis} in the negatively curved setting, this invariant associated to a discrete group of isometries has been object of numerous deep results. Recall for example the Patterson-Sullivan theory used for precise orbital counting, or its rigid structure due to Besson-Courtois-Gallot \cite{BCG}, Bowen \cite{bowen-quasicircles} and Bourdon \cite{bourdon}, just to name a few.

This paper is concerned on the rigidity problem for Hitchin
representations (the orbital counting problem has already been
treated in \cite{orbitalcounting}). Normalize $d_X$ so that the
totally geodesic embedding of $\H^2$ in $X,$ induced by the
morphism $\PSL(2,\R)\to\PSL(d,\R)$ has curvature $-1.$ The main
result of this work is the following theorem.

\begin{teo1} For all $\rho\in\Hitchin(\E,d)$ one has $h_X(\rho(\pi_1\E))\leq 1$ and equality only holds if $\rho$ is Fuchsian.
\end{teo1}

Theorem A confirms the current philosophy that deformations in higher rank spaces should decrease the critical exponent, as opposed to deformations on rank 1 spaces (i.e. pinched negative curvature) where the critical exponent increases (see Bowen's fundamental paper \cite{bowen-quasicircles} on quasi-Fuchsian representations). It would be interesting to find a global explanation for these two different phenomena, today understood independently: in rank 1 the critical exponent is the Hausdorff dimension of the limit set, bounded below by the topological dimension; in higher rank (as we shall see below) it is the possibility of growing in different directions that forces $h_X$ to decrease.

This philosophy probably originated in Bishop--Steger's work \cite{bishop}, where they show that if $\rho,\eta\in\Hitchin(\E,2)$ then $$h^{(1,1)}(\rho, \eta)=\lim_{s\to\infty}\frac{\log\#\{[\g]\in[\pi_1\E]:|\rho\g|+|\eta\g|\leq s\}}s\leq 1/2,$$ where $|g|$ is the translation distance of $g$ in $\H^2$ and $[\pi_1\E]$ denotes the set of conjugacy classes of $\pi_1\E.$ Moreover, equality implies $\rho=\eta.$ As noticed by Burger \cite{burger}, this is a rank-2 problem, associated to the product representation $\rho\times\eta:\pi_1\E\to\PSL(2,\R)\times\PSL(2,\R).$

An analogous result holds for \emph{Benoist representations}\footnote[2]{These are also called \emph{divisible convex sets with strictly convex boundary,} or \emph{strictly convex projective structures on closed manifolds.}}. These are homomorphisms $\rho:\G\to\PGL(n+1,\R)$ where $\G$ is a word-hyperbolic group, such that $\rho(\G)$ preserves an open convex set $\Om\subset\P(\R^{n+1})$ properly contained on an affine chart, and such that the quotient $\rho(\G)\/\Om$ is compact. The Hilbert metric on $\Om$ induces a $\rho(\G)$-invariant Finsler metric on $\Om.$ Crampon \cite{crampon} proved that the topological entropy of the geodesic flow on ${\sf{T}}^1\rho(\G)\/\Om$ associated to this metric, is bounded above by $n-1$ and equality only holds if $\Om$ is an ellipsoid. We provide a new proof of Crampon's result in Section \ref{section:convex}.

It is consequence of Choi-Goldman's work \cite{choigoldman} that the space of Benoist representations of $\pi_1\E$ coincides with $\Hitchin(\E,3).$

Before explaining the main ideas of the proof let us remark that, as explained by Labourie \cite[Section 1.4]{labourie-cyclicsurfaces}, the inequality in Theorem A implies a (rigid) inequality concerning the area of a minimal surface on $\rho(\pi_1\E)\/X.$ Recall from Labourie \cite{labourie-crossratioENS} that the \emph{minimal area} of $\rho$ is defined by $$\aire(\rho)=\inf\{e_\rho(J):J\in\Hitchin(\E,2)\},$$ where $e_\rho(J)$ is the energy of the unique harmonic $\rho$-equivariant map from $\E$ e\-quipped with $J$ to $\rho(\pi_1\E)\/X.$ It follows from Hitchin's construction that such a harmonic map is an immersion (see Sanders \cite{andy} for details). Standard computations imply that the metric induced on this immersed surface is necessarily negatively curved and hence its topological entropy is bounded above by 
$h_X(\rho(\pi_1\E)).$ Applying a theorem of Katok \cite[Theorem B]{katok1982} one has $$\aire(\rho)\geq\frac{-2\pi\chi(\E)}{h^2_X(\rho(\pi_1\E))},$$ where $\chi(\E)$ is the E\"uler characteristic of $\E.$ Consequently, Theorem A implies the following:

\begin{cor} Let $\rho\in\Hitchin(\E,d)$ then $$\aire(\rho)\geq-2\pi\chi(\E)$$ and equality only holds if $\rho$ is Fuchsian.
\end{cor}

This is a theorem of Labourie \cite[Theorem 1.4.1]{labourie-cyclicsurfaces} when the Zariski closure of $\rho$ has rank 2, proved using Higgs bundles techniques.

Finally, let us note that Theorem A is still open for the Hitchin components of the real split simple groups $\PSO(n,n)$ ($n\geq4$) and the exceptional real split Lie groups (except $\Ge$). This is due to the fact that the Frenet property of Labourie's equivariant flag curve (see below) is only known to hold for $\Hitchin(\E,d)$ (and hence for the groups $\Psp(2k,\R),$ $\PSO(k,k+1)$ and $\Ge,$ since their respective Hitchin components are canonically embeded in $\Hitchin(\E,d)$ for $d=2k,$ $2k+1$ and $7$ respectively).

\subsection{Proof of Theorem A: The asymptotic location of eigenvalues} The general method is not specific to the Hitchin component. Indeed, our method applied in different situations gives an improvement of Crampon's result and a generalization of Bishop--Steger's theorem to arbitrary products such as $$\Hitchin(\E,d_1)\times \cdots\times\Hitchin(\E,d_k),$$ replacing $1/2$ with a proper upper bound. We will explain here how the idea works in the Hitchin component, and leave to Section \ref{section:convex} the case of Benoist's representations.

The first step of the proof of Theorem A reposes on some previous results of Quint \cite{quint2} and S. \cite{exponential} which relate the critical exponent with the (asymptotic) location of the eigenvalues of a Hitchin representation.

Let $\frak a=\{a\in\R^d:a_1+\cdots+a_d=0\}$ be a Cartan subalgebra of $\frak{sl}(d,\R)$ and denote by $\eps_i(a)=a_i.$ Let $$\frak a^+=\{a\in\frak a: a_1\geq\cdots\geq a_d\}$$ be a closed Weyl chamber and $\Pi=\{\sigma_i=\eps_i-\eps_{i+1}\in\frak a^*:i\in\{1,\ldots,d-1\}\}$ the set of simple roots associated to the choice of $\frak a^+.$ Denote by $\lambda:\PSL(d,\R)\to\frak a^+$ the \emph{Jordan projection}: $$\lambda(g)=(\lambda_1(g),\ldots,\lambda_d(g)),$$ consisting on the $\log$ of the modulus of the eigenvalues of $g$ (possibly with repetition) and in decreasing order.

For $\rho\in\Hitchin(\E,d)$ denote by $\cone_\rho$ the closed cone of $\frak a^+$ generated by $\{\lambda(\rho\g):\g\in\pi_1\E\}.$ This cone contains all possible directions where $\lambda(\rho(\pi_\1\E))$ is. A finer invariant is to understand \emph{how many} eigenvalues of $\rho$ are on a given direction inside $\cone_\rho.$ Denote by $\cone_\rho^*=\{\varphi\in\frak a^*:\varphi|\cone_\rho\geq0\}$ the \emph{dual cone} of $\cone_\rho.$ For $\varphi\in \cone_\rho^*$ define its \emph{entropy} by $$h^\varphi_\rho=\lim_{s\to\infty}\frac{\log\#\{[\g]\in[\pi_1\E]:\varphi(\lambda(\rho\g))\leq s\}}s.$$

A linear form $\varphi$ belongs to the interior of $\cone_\rho^*$ if and only if $h^\varphi_\rho$ is finite and positive (Lemma \ref{lema:interior}). The main object we are interested in is the set $$\D_\rho=\{\varphi:h^\varphi_\rho\in(0,1]\}.$$ Proposition \ref{prop:convexoreps} states that $\D_\rho$ is a convex subset of $\frak a^*,$ and the formula $h^{t\varphi}_\rho=h^\varphi_\rho/t$ implies that if $\varphi\in\D_\rho$ then $t\varphi\in\D_\rho$ for all $t\geq 1.$ Moreover, its boundary $\bord\D_\rho=\{\varphi:h^\varphi_\rho=1\}$ is a codimension 1 closed analytic submanifold of $\frak a^*.$ The shape of $\D_\rho$ will be crucial in the sequel.

Recall that $d_X$ is a distance on $X$ induced by a
$\PSL(d,\R)$-invariant Riemannian metric on $X.$ Denote by $\|\
\|_{\frak a}$ the Euclidean norm on $\frak a$ (invariant under the
Weyl group) induced by $d_X,$ and by $\|\ \|_{\frak a^*}$ the
induced norm on $\frak a^*.$ One has the following
result\footnote[2]{Proposition \ref{prop:JF} actually holds on a much more general setting, see subsection \ref{subsection:hist}.}.

\begin{prop}[{Quint \cite[Corollary 3.1.4]{quint2} + \cite[Corollary 4.4]{exponential}}]\label{prop:JF} Let $\rho\in\Hitchin(\E,d)$ then $$h_X(\rho(\pi_1\E))=\min\{\|\varphi\|_{\frak
a^*}:\varphi\in\D_\rho\}.$$
\end{prop}

\begin{ex}\label{exFuchs}The irreducible linear action $\tau_d:\PSL(2,\R)\to\PSL(d,\R)$ is given by the canonical action of $\PSL(2,\R)$ on the $(d-1)$-symmetric power ${\sf{S}}^{d-1}(\R^2)$ of $\R^2.$ If $g\in\PSL(2,\R)$ one has $\lambda_{\PSL(2,\R)}(g)=(|g|/2,-|g|/2),$ where $|g|$ denotes the translation distance of $g,$ and hence $$\lambda(\tau_dg)=\frac{|g|}2(d-1,d-3,\cdots,3-d,1-d).$$ Thus, for all $\sigma\in\Pi$ one has $\sigma(\lambda(\tau_dg))=|g|.$ Moreover if $\varphi$ belongs to the affine hyperplane generated by $\Pi,$ $$V_\Pi=\{\sum_{\sigma\in\Pi} t_\sigma\sigma:\sum t_\sigma=1\},$$ then $\varphi(\lambda(\tau_dg))=|g|.$ Consequently, if $\rho_0\in\Hitchin(\E,d)$ is Fuchsian then $\bord\D_{\rho_0}=V_\Pi.$ Since $d_X$ is normalized such that the totally geodesic embedding of $\mathbb{H}^2$ in $X$ to have curvature $-1,$ the Fuchsian representation $\rho_0$ has critical exponent equal to $1.$ One concludes, using Proposition \ref{prop:JF}, that $$\min\{\|\varphi\|_{\frak a^*}:\varphi\in V_\Pi\}=1$$ and this minimum is realized in the dual space of the Cartan algebra $$\{(d-1,d-3,\ldots,3-d,1-d)t:t\in\R\}$$ of $\tau_d(\PSL(2,\R)).$
\end{ex}

The proof of Theorem A consists in a deeper understanding of the set $\D_\rho$ for a given $\rho\in\Hitchin(\E,d),$ and its relative position with respect to $V_\Pi.$

Denote by $G=G_\rho$ the Zariski closure of $\rho(\pi_1\E). $ The group $G$ is necessarily semisimple\footnote[4]{It is reductive, since it acts irreducibly on $\R^d$ (Labourie \cite[Lemma 10.1]{labourie}) and has no center, since moreover $\forall \g\in\pi_1\E,$ $\rho(\g)$ is proximal (see Benoist \cite{convexes3}).}. Choose a Cartan subalgebra $\frak a_G\subset \frak a$ and a Weyl chamber $\frak a_G^+\subset\frak a^+. $ Consider the restriction map $\rest:\frak a^*\to\frak a_G^*,$ defined by $\rest(\varphi)=\varphi|\frak a_G.$ Observe that, since the vector space spanned by $\{\lambda(\rho\g):\g\in\pi_1\E\}$ is $\frak a_G,$ the entropy of a given linear form $\varphi,$ is the entropy of $\rest(\varphi).$

Remark \ref{obs:non-arithmetic} and Proposition \ref{prop:convexoreps} below imply that $\rest(\D_\rho)$ is strictly convex. Since $\|\ \|_{\frak a}$ is Euclidean one can (and will) identify the space $\frak a_G^*$ with a subspace of $\frak a^*.$ Namely, denote by $p_G:\frak a\to\frak a_G$ the orthogonal projection, then $$\frak a^*_G=\{\varphi\in\frak a^*: \varphi=\varphi\circ p_G.\}.$$

The set $\D_\rho$ is hence a convex set, whose intersection with $\frak a^*_G$ is strictly convex (see figure \ref{figure-DrhoinFraka}).

\begin{figure}[ht]\begin{center}
\input{drho2.pstex_t}
\caption{\small{The set $\D_\rho$ when $\frak a_G^\ast$ is a
strict subspace of $\frak a^\ast$.}}\label{figure-DrhoinFraka}
\end{center}\end{figure}

The second important step in the proof of Theorem A is the following theorem, its statement arose from an insightful discussion between the second author with Bertrand Deroin and Nicolas Tholozan.

\begin{teoB} For every $\rho\in\Hitchin(\E,d)$ and $\sigma\in\Pi$ one has $h^\sigma_\rho=1.$
\end{teoB}

Theorem B states that the simple roots $\sigma$ always belong to $\bord\D_\rho,$ regardless of $\rho\in\Hitchin(\E,d).$  Let us explain how this implies Theorem A.

\begin{proof}[Proof of Theorem A] Let $\Delta_\Pi$ be the convex hull of $\Pi,$ denote by $\inte\Delta_\Pi$ its relative interior and consider $\rho\in\Hitchin(\E,d).$ Since $\D_\rho$ is convex and $\Pi\subset\bord\D_\rho$ one has $\Delta_\Pi\subset\D_\rho.$ Hence, Proposition \ref{prop:JF} and the computations in Example \ref{exFuchs} give $$h_X(\rho)=\min\{\|\varphi\|_{\frak a^*}:\varphi\in\D_\rho\}\leq\min\{\|\varphi\|_{\frak a^*}:\varphi\in\Delta_\Pi\}=1.$$

If $h_X(\rho)=1,$ then the intersection $\bord\D_\rho\cap\inte\Delta_\Pi$ is non-empty, thus $\inte\Delta_\Pi\subset\bord\D_\rho.$ Moreover, since $\bord\D_\rho$ is closed one has $\Delta_\Pi\subset\bord\D_\rho.$

Since $\D_\rho\cap \frak a^*_G$ is strictly convex, the only
possibility is for $\frak a^*_G$ to be $1$-dimensional, i.e. the
Zariski closure of $\rho$ has rank 1\footnote[2]{A recent classification of possible Zariski closures of a Hitchin representation, obtained by Guichard \cite{clausura}, implies directly that $G_\rho$ is isomorphic to $\PSL(2,\R).$ In our present situation a direct proof of this fact is possible and easy, so we include it for completeness.}.  Moreover, $\frak
a_G=\{(d-1,d-3,\cdots,1-d)t:t\in\R\}.$ Since a purely loxodromic
matrix does not commute with a one-parameter compact group,
$G_\rho$ is simple and actually its Lie algebra is isomorphic to $\frak{sl}(2,\R)$ (recall
the classification of rank 1 real-algebraic simple Lie groups). Hence, the group $G_\rho$ is a finite covering of $\PSL(2,\R).$ Since $G_\rho$ is linear the connected component of the identity ${(G_\rho)}_0$ is isomorphic to $\PSL(2,\R).$ Since $\rho$ can be connected to a Fuchsian representation, for every $\g\in\pi_1\E$ there exists a path, through purely loxodromic matrices, from $\rho(\g)$ to a diagonalizable matrix with eigenvalues of the same sign. This implies that $\rho(\g)$ has all its eigenvalues of the same sign and hence belongs to $(G_\rho)_0.$
This completes the proof.
\end{proof}

\begin{figure}[ht]\begin{center}
\input{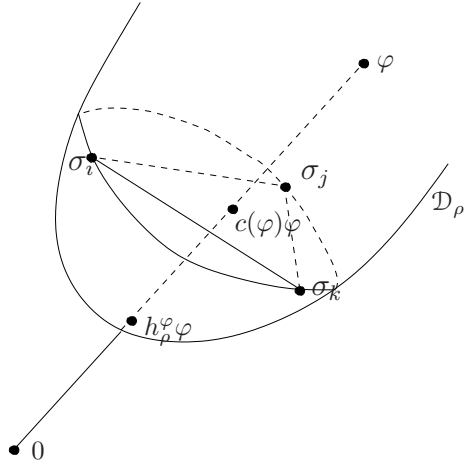}
\caption{\small{The simple roots force the linear form in $\D_\rho$ closest to the origin, to be below a
certain affine subspace.}}\label{figure-Drho}
\end{center}\end{figure}

In fact, Theorem B and the last proof provide a rigid upper bound for the entropy of each linear form in the interior of the dual cone $(\frak a^+)^*.$ Indeed, if $\varphi\in\inte(\frak a^+)^*$ then it is a linear combination of elements in $\Pi$ with (strictly) positive coefficients, i.e. the half line $\R_+\cdot\varphi$ intersects $\inte\Delta_\Pi.$ Notice that $h^\varphi_\rho$ is the only number such that $$h_\rho^\varphi\varphi\in\bord\D_\rho.$$ The upper bound of $\rho\mapsto h_\rho^\varphi$ is hence the number $\algo{\varphi}$ such that $\algo{\varphi}\varphi\in\Delta_\Pi$ (see figure \ref{figure-Drho}).

\begin{cor}\label{cor:cfi} Consider $\varphi\in\inte(\frak a^+)^*,$ then for all $\rho\in\Hitchin(\E,d)$ one has $h^\varphi_\rho\leq \algo{\varphi},$ and equality only holds if $\rho$ is Fuchsian.
\end{cor}

In particular, considering the linear form $\varphi_{1d}(a)=(a_1-a_d)/2=(\sum\sigma_i)/2,$ one has $h^{\varphi_{1d}}_\rho\leq 2/(d-1).$  Also, notice that $\varphi_1(a)=a_1 = \frac{1}{d} \sum_{j=1}^{d-1} (d-j)\sigma_j$ therefore one also has $c(\varphi_1)=2/(d-1).$

In \cite[Corollary 3.4]{entropia} a similar inequality is proved, namely $\alpha h^{\varphi_1}_\rho\leq 2/(d-1),$ where $\alpha$ is the H\"older exponent of Labourie's equivariant flag curve (see below) for a visual metric on $\bord\pi_1\E$ (induced by a choice of a hyperbolic metric on $\E$). These two rigid inequalities are different in nature: while equality in Corollary \ref{cor:cfi} implies that a totally geodesic copy of $\H^2$ is preserved, \cite[Corollary 3.4]{entropia} states that equality in $\alpha h^{\varphi_1}_\rho\leq 2/(d-1)$ recognizes a specific representation in $\tau_d(\PSL(2,\R)).$

It is interesting to remark that the same argument shows the existence of linear forms whose entropy is bounded from \emph{below} (when defined). For example: $(1+\eps_1)\sigma_1 -\sum_2^d\eps_i \sigma_i$ for small enough $\eps_i>0$ works. 

Furthermore, the special shape of $\bord\D_\rho$ actually provides a 'simple' criterion to determine the rank of the Zariski closure of a Hitchin representation. Observe that $\Delta_\Pi$ is a $(d-1)$-dimensional simplex. Let $F_k\subset\Delta_\Pi$ be a $k$-dimensional face and denote by $\inte F_k$ its relative interior.

\begin{cor}\label{cor:clausura} Consider $\rho\in\Hitchin(\E,d)$ and assume that $(\inte F_k)\cap\bord\D_\rho\neq\vacio,$ then $\rk(G_\rho)\leq \dim\frak a-k.$
\end{cor}

\begin{proof} As in the proof of Theorem A, the fact that $(\inte
F_k)\cap\bord\D_\rho\neq\vacio$ implies that
$F_k\subset\bord\D_\rho.$ Since $\bord\D_\rho$ is a closed
analytic submanifold of $\frak a$ (Proposition
\ref{prop:convexoreps}), one concludes that the affine space
$V_{F_k}$ spanned by $F_k$ is contained in $\bord\D_\rho.$

Recall that $\D_\rho\cap \frak a_{G_\rho}^*$ is strictly convex,
thus $\frak a_{G_\rho}^*$ is transverse to a $k$-dimensional
affine space. Hence $ \dim \frak a_{G_\rho} +k\leq\dim \frak a.$
This finishes the proof.
\end{proof}

\subsection{Theorem B: Finding a suitable Anosov flow} The proof of Theorem B is based on the following (SRB)-principle (Corollary \ref{cor:rep1}): If $\phi$ is a $\clase^{1+\alpha}$ Anosov flow on a closed manifold $X,$ and $\lambda^u:X\to\R_+$ denotes the infinitesimal expansion rate in the unstable direction, then the reparametrization of $\phi$ by $\lambda^u$ has topological entropy equal to 1.

The proof of Theorem B goes by finding, for each $i\in\{2,\ldots,d-1\},$ an Anosov flow whose periodic orbits are indexed in $[\pi_1\E],$ such that the total expansion rate along the periodic orbit $[\g]\in[\pi_1\E]$ is given by $$\int_{[\g]}\lambda^u=\sigma_{i-1}(\lambda(\rho\g)).$$

In $\Hitchin(\E,d)$ our construction only works locally, i.e. on a
neighborhood of the Fuchsian locus, nevertheless the construction is global in
the Hitchin components of the groups $\Ge,$ $\Psp(2k,\R)$ and
$\PSO(k,k+1).$ Analyticity of the entropy function
will allow us to conclude Theorem B in the whole component $\Hitchin(\E,d).$

A basic tool for understanding Hitchin representations is \emph{Labourie's \cite{labourie} equivariant flag curve}.

Let $\scr F$ be the space of complete flags of $\R^d,$ then given $\rho\in\Hitchin(\E,d)$ there exists an equivariant H\"older-continuous map $\z=\z(\rho):\bord\pi_1\E\to\scr F.$ One denotes by $\z_i(x)$ the $i$-dimensional subspace of $\R^d$ associated to $\z(x).$

This equivariant map is a \emph{Frenet curve}, i.e. for
every decomposition $n=d_1+\cdots+d_k\leq d$ ($d_i\in\N$), and
$x_1,\ldots,x_k\in\bord\pi_1\E$ pairwise distinct, the subspaces
$\z_{d_i}(x_i)$ are in direct sum, and moreover
$$\lim_{(x_i)\to x}\bigoplus_1^k\z_{d_i}(x_i)=\z_n(x).$$
This condition implies that one can recover $\z$ from $\z_1$ and we shall sometimes call $\z_1$ the \emph{Frenet equivariant curve} of $\rho$ too. 

The existence of this curve guarantees that each $\rho\g$ is
diagonalizable, indeed, if $\g_+$ and $\g_-$ are the attracting
and repelling points of $\g$ on $\bord\pi_1\E,$ then for
$i\in\{1,\ldots, d\}$ one has that
$$\ell_i(\g_+,\g_-) =\z_i(\g_+)\cap\z_{d-i+1}(\g_-)$$
is a $\rho\g$-invariant line, and its associated eigenvalue has
modulus $e^{\lambda_i(\rho\g)}.$ The Frenet condition implies that the
projective trace of $\z,$ i.e. $\z_1(\bord\pi_1\E),$ is a
$\clase^{1}$-submanifold of $\P(\R^d).$

Denote by $\bord^2\pi_1\E=\{(x,y)\in(\bord\pi_1\E)^2:x\neq y\}.$ We prove in Proposition \ref{prop:ith-eigenvalue} that the
function $\ell_i:\bord^2\pi_1\E\to\P(\R^d)$ defined by $$\ell_i(x,y):=\z_i(x)\cap\z_{d-i+1}(y),$$ provides a
$\clase^{1+\alpha}$ submanifold of $\P(\R^d),$ namely
$$\Lim^i_\rho :=\{\ell_i(x,y):(x,y)\in\bord^2\pi_1\E\}.$$ Moreover when $i=2,\ldots,d-1,$ the tangent space
$T_{\ell_i(x,y)}\Lim^i_\rho$ splits as $$\hom(\ell_i(x,y),\ell_{i-1}(x,y))\oplus\hom(\ell_i(x,y),\ell_{i+1}(x,y)).$$

Consider now the bundle $\w{\sf{F}}^i_\rho$ over $\Lim^i_\rho$
whose fiber ${\sf{M}}^i_\rho(x,y)$ at $\ell_i(x,y)$ consists on
the elements of $\ell_i(x,y),$ i.e.
$${\sf{M}}^i_\rho(x,y)=\{v\in\ell_i(x,y)-\{0\}\}/v\sim-v.$$ The
fiber bundle $\w{\sf{F}}^i_\rho$ is equipped with the action of
$\rho(\pi_1\E)$ and with a commuting $\R$-action, defined on each
fiber by $$\w\phi^i_t(v)=e^{-t}v.$$

\begin{teoC}There exists a neighborhood $U$ of the Fuchsian locus on $\Hitchin(\E,d),$ such that if $\rho\in U$ then, for every $i\in\{2,\ldots,d-1\}$ with $i\neq(d+1)/2,$ the action of $\rho(\pi_1\E)$ on $\w{\sf{F}}^i_\rho$ is properly discontinuous and cocompact. The flow $\phi^i$ induced on the quotient ${\sf{F}}^i_\rho=\rho(\pi_1\E)\/\w{\sf{F}}^i_\rho$ is a $\clase^{1+\alpha}$ Anosov flow, whose unstable distribution is given by $E^u_i=\hom(\ell_i(x,y),\ell_{i-1}(x,y)).$
\end{teoC}

Theorem C is the statement of Corollary \ref{cor:tuttipapel}. Sections \ref{section:i} and \ref{section:cociclosholder} are devoted to its proof.

\begin{ex}
When $d=3$ the representation $\rho$ preserves a proper open
convex set $\Omega\subset \P(\R^3)$ and the map $\ell_2$ is a
$2$-fold covering from the annulus $\bord^2\Omega$ to the M\"obis
strip $\P(\R^3)-\overline\Omega$ (see Barbot \cite{barbot}).  If
moreover $\rho\in\Hitchin(\E,3)$ is Fuchsian, then
$\lambda_2(\rho\g)=0$ for all $\g\in\pi_1\E,$ hence each $v\in
{\sf{M}}^2_\rho(\g_+,\g_-)$ is fixed by $\rho\g.$ Thus, the action
of $\rho(\pi_1\E)$ on $\w{\sf{F}}^i_\rho$ is not proper. A similar situation
occurs for $d=2k-1$ and $i=k.$
\end{ex}

\begin{obs}\label{obs:U} The neighborhood $U$ of Theorem C is explicit. For $i\in\{2,\ldots,d-1\}$ denote by $$U_i=\{\rho\in\Hitchin(\E,d): \cone_\rho\cap\ker \eps_i=\{0\}\}$$ ($U_1$ and $U_d$ are uninteresting since $\frak a^+\cap\ker\eps_1=\frak a^+\cap\ker\eps_d=\{0\}$). This is an open set (Corollary \ref{cor:conolimite}) that contains the Fuchsian locus except when $d=2k-1$ and $i=k.$ Theorem C is proved for $U=\bigcap_{i\neq(d+1)/2} U_i.$ Notice that the case $\Hitchin(\E,3)$ needs to be treated separately, we do so in section \ref{section:convex}.
\end{obs}

Assume from now on that $d\neq3$ and that $i\neq (d+1)/2.$ Let $U$ be the neighborhood provided by Theorem C and consider $\rho\in U.$ Since $\phi^i$ is a $\clase^{1+\alpha}$ Anosov flow, one can consider the expansion rate $\lambda^u:{\sf{F}}^i_\rho\to\R_+$ along the unstable distribution $E^u_i$ defined by $$\lambda^u(x)=\left.\frac{\partial}{\partial t}\right|_{t=0}\frac 1 \kappa
\int_0^\kappa \log \det(d_x\phi_{t+s}^i|E^u_i) ds$$ (for any $\kappa>0,$ see subsection
\ref{subsection:1u}). Corollary \ref{cor:tuttipapel} states that if $\g\in\pi_1\E$ then
$$\int_{\g}\lambda^u=\sigma_{i-1}(\lambda(\rho\g)),$$ i.e. if one reparametrizes $\phi^i$ with $\lambda^u,$ then the period of the periodic orbit
$[\g]$ is $\sigma_{i-1}(\lambda(\rho\g)).$

Corollary \ref{cor:rep1} states that the reparametrization of $\phi^i$ by $\lambda^u $ has topological entropy $1.$ Since the topological entropy of an Anosov flow is the exponential growth rate of its periodic orbits, one concludes $$1=\lim_{s\to\infty}\frac{\log\#\{[\g]\in[\pi_1\E]:\sigma_{i-1}(\lambda(\rho\g))\leq s\}}s=h^{\sigma_{i-1}}_\rho.$$

The unstable distribution of the inverse flow $v\mapsto
\phi^i_{-t}v,$ is $\hom(\ell_i(x,y),\ell_{i+1}(x,y)),$ so the same
argument proves that $h^{\sigma_{i}}_\rho=1.$ Finally, observe that
even though $i\neq(d+1)/2,$ we have achieved all possible simple
roots.

One concludes that for all $\sigma\in\Pi,$ the function $\rho\mapsto h_\rho^\sigma$ is constant equal $1$ on the open set $U.$ Since $\Hitchin(\E,d)$ is an analytic manifold (Hitchin \cite{hitchin}), Corollary \ref{cor:conolimite} implies that this map is analytic on $\Hitchin(\E,d),$ hence, it is globally constant. This finishes the proof of Theorem B.

\subsection{Further consequences}\label{subsect:furthercons}

Labourie \cite{labourie} observes that if $\rho\in\Hitchin(\E,d)$ and its equivariant Frenet curve $\z_1:\bord\pi_1\E\to\P(\R^d)$ is of class $\clase^\infty$ then one can recover the flag curve by means of its derivatives, namely $$\z_k=\z_1\oplus\z_1'\oplus\cdots\oplus \z_1^{(k-1)},$$ where $\z_1^{(i)}$ is the $i$-th derivative of $\z_1$ in an affine chart. He also remarks that there is no reason for $\z_1$ to be of class $\clase^\infty,$ we prove in section \ref{section:regularidad} the following theorem.

\begin{teoD} Let $\rho$ be a Hitchin representation such that $\z_1$ is of class $\clase^\infty,$ then $\rho$ is Fuchsian.
\end{teoD}


\subsection{Historical comments}\label{subsection:hist}

A slightly different version of the set $\D_\rho$ was introduced by Burger \cite{burger} for product representations  $\rho=\rho_1\times\rho_2:\G\to G_1\times G_2,$ where $G_i$ is a simple rank 1 group, and $\rho_i:\G\to G_i$ is convex cocompact. It is also dual to Quint's \cite{quint2} \emph{growth indicator function}, defined for a Zariski-dense subgroup of a real-algebraic semisimple Lie group. Quint's definition involves the Cartan projection (instead of the Jordan projection) and with his definition Proposition \ref{prop:JF} holds for any such subgroup (Quint \cite{quint2}). The relation between our definition and his, established in \cite{exponential}, (is only known to) holds for a Anosov representation of a hyperbolic group with respect to a minimal parabolic subgroup.

The statement of Theorem B arose from a discussion between the second author with Bertrand Deroin and Nicolas Tholozan. Using random walk techniques, they prove \cite{BertrandNicholas} that if $\rho,\eta\in\Hitchin(\E,d)$ and $\sigma\in\Pi$ then $$\sup_{\g\in\pi_1\E}\frac{\sigma(\rho\g)}{\sigma(\eta\g)}\geq1.$$ Their theorem suggested that Theorem B should be true and it is quite possible that their method also provides a proof.

The construction of the flow $\phi^i=(\phi_t^i:\sf F^i_\rho\to \sf F^i_\rho)_{t\in\R}$ is analogous to the construction of the geodesic flow of a projective Anosov representation in \cite{presion}, this construction is explained in section \ref{section:convexA}. The advantage of considering this variation is that one can guarantee further regularity of the objects on consideration, which is needed to apply the Sinai-Ruelle-Bowen Theorem. The geodesic flow of a projective Anosov irreducible representation was introduced in \cite{quantitative} under the terminology of convex representations.

\subsection{Acknowledgements}
Theorem A is a question due (independently?) to Gilles Courtois
and Fran\c cois Labourie, A. Sambarino is grateful to them for
proposing this problem. He is also extremely grateful to Bertrand Deroin and Nicolas Tholozan for insightful discussions, in particular for suggesting that $h^{\sigma}_\rho=1$ should be true. He would also like to thank Marc Burger, Olivier Glorieux and
Fran\c{c}ois Labourie for interesting discussions. R. Potrie
gratefully acknowledges the hospitality of Universit\'e Paris-Sud
(Orsay).


\section{Reparametrizations and Thermodynamic Formalism}\label{section:1}


Let $X$ be a compact metric space, $\phi=(\phi_t)_{t\in\R}$ a continuous flow on $X$ without fixed points and $V$ a finite dimensional real vector space. Consider a continuous map $f:X\to V,$ and denote by $p(\tau)$ the period of a $\phi$-periodic orbit $\tau.$ The \emph{period} of $\tau$ for $f$ is defined by $$\int_\tau f=\int_0^{p(\tau)}f(\phi_sx)ds,$$ for any $x\in\tau.$

We say that a map $U:X\to V$ is $\clase^1$ \emph{in the direction
of the flow} $\phi,$ if for every $x\in X,$ the map $t\mapsto
U(\phi_tx)$ is of class $\clase^1,$ and the map

$$x\mapsto \left.\frac{\partial }{\partial t}\right|_{t=0}U(\phi_tx)$$ is
continuous. Two continuous maps, $f,g:X\to V$ are \emph{Liv\v
sic-cohomolo\-gous} if there exists a map $U,$ which is $\clase^1$
in the direction of the flow, such that for all $x\in X$ one has

$$f(x)-g(x)=\left.\frac{\partial}{\partial t}\right|_{t=0}
U(\phi_tx).$$ Notice that if this is the case then $\int fdm=\int
gdm$ for any $\phi$-invariant measure $m.$ In particular, $f$ and
$g$ have the same periods.

If $f:X\to\R$ is positive, then $f$ has a positive minimum and
hence for every $x\in X$ the function $\k_f:X\times\R\to V,$
defined by $\k_f(x,t)=\int_0^tf(\phi_sx)ds,$ is an increasing
homeomorphism of $\R.$ Thus there is a continuous function
$\alpha_f:X\times\R\to\R$ that verifies

\begin{equation}\label{equation:inversa}
\alpha_f(x,\k_f(x,t))=\k_f(x,\alpha_f(x,t))=t,\end{equation} for
every $(x,t)\in X\times\R.$

\begin{defi}\label{defi:repa}The \emph{reparametrization} of $\phi$ by $f:X\to\R_{>0},$ is the flow $\psi=\psi^f=(\psi_t)_{t\in\R}$ on X defined by $\psi_t(x)=\phi_{\alpha_f(x,t)}(x),$ for all $t\in\R$ and $x\in X.$ If $f$ is H\"older-continuous, we say that $\psi$ is a H\"older reparametrization of $\phi.$
\end{defi}

By definition, the period of a periodic orbit $\tau$ for $\psi^f$
is the period of $\tau$ for $f.$ Denote by $\cal M^{\phi}$ the
space of $\phi$-invariant probability measures on $X.$ The
\emph{pressure} of a continuous function $f:X\to\R,$ is defined by
$$P(\phi,f)=\sup_{m\in\cal M^{\phi}}h(\phi,m)+\int_X fdm,$$ where
$h(\phi, m)$ is the metric entropy of $m$ for $\phi.$ A
probability measure $m,$ on which the least upper bound is
attained, is called an \emph{equilibrium state} of $f.$ An
equilibrium state for $f\equiv0$ is called a \emph{measure of
maximal entropy}, and its entropy is called the \emph{topological
entropy} of $\phi,$ denoted by $h_{\textrm{top}}(\phi).$


\begin{lema}[{\cite[Lemma 2.4]{quantitative}}]\label{lema:reparam} Let $f:X\to\R_{>0}$ be a continuous function. Assume the equation $$P(\phi,-sf)=0\qquad s\in\R,$$ has a finite positive solution $h,$ then $h$ is the topological entropy of $\psi^f.$ In particular the solution is unique. Conversely if $h_{\textrm{\emph{top}}}(\psi^f)$ is finite then it is a solution to the last equation.
\end{lema}

\subsection{(Metric) Anosov flows and vector valued potentials}

We will now define \emph{metric Anosov flows}. The transfer of
classical results from axiom A flows to this more general setting
is provided by Pollicott's work \cite{smaleflows}, and references
therein.

As before $\phi$ denotes a continuous flow on the compact metric space $X.$ For $\eps>0$ one defines the \emph{local stable set} of $x$ by $$W^s_\eps(x)=\{y\in X:d(\phi_tx,\phi_t y)\leq\eps\ \forall t>0\textrm{ and }d(\phi_tx,\phi_t y)\to0\textrm{ as }t\to\infty\}$$ and the \emph{local unstable set} by $$W^u_\eps(x)=\{y\in X:d(\phi_{-t}x,\phi_{-t} y)\leq\eps\ \forall t>0\textrm{ and }d(\phi_{-t}x,\phi_{-t} y)\to0\textrm{ as }t\to\infty\}.$$

\begin{defi}\label{defi:anosovtop} We will say that $\phi$ is a \emph{metric Anosov flow} if the following holds:
\begin{itemize} \item[-] There exist positive constants $C,$ $\lambda$ and $\eps$ such that for every $x\in X,$ every $y\in W^s_\eps(x)$ and every $t>0$ one has $$d(\phi_t(x),\phi_t(y))\leq Ce^{-\lambda t}$$ and such that for every $y\in W^u_\eps(x)$ one has $$d(\phi_{-t}(x),\phi_{-t}(y))\leq Ce^{-\lambda t}.$$
\item[-] There exists $\delta>0$ and a continuous map $\nu:\{(x,y)\in X\times X:d(x,y)<\delta\}\to \R$ such that $\nu(x,y)$ is the unique value such that $W^u_\eps(\phi_\nu x)\cap W^s_\eps(y)$ is non empty, and consists of exactly one point.
\end{itemize}
\end{defi}

A flow is said to be \emph{transitive} if it has a dense orbit. From now on we will assume that $\phi$ is a transitive metric Anosov flow.



\begin{teo}[Liv\v sic \cite{livsic}]\label{teo:livsic3} Consider a H\"older-continuous map $f:X\to V,$ if $\int_\tau f=0$ for every periodic orbit $\tau,$ then $f$ is Liv\v sic-cohomologous to $0.$
\end{teo}

Consider a H\"older-continuous function $f:X\to \R$ with
non-negative periods and define its \emph{entropy} by
$$h_f=\limsup_{s\to\infty}\frac1s\log\#\{\tau\textrm{
periodic}:\int_\tau f\leq s\}\in[0,\infty].$$ Clearly, the entropy
of a function only depends on the periods of the function, therefore
two Liv\v sic cohomologous functions have the same entropy. One
has the following lemma.

\begin{lema}[{Ledrappier \cite[Lemma 1]{ledrappier}+\cite[Lemma 3.8]{quantitative}}]\label{lema:positiva} Consider a H\"older-continuous function $f:X\to\R$ with non-negative periods. Then the following statements are equivalent: \begin{itemize}\item[-] the function $f$ is Liv\v sic-cohomologous to a positive H\"older-continuous function,  \item[-] there exists $\kappa>0$ such that $\int_\tau f>\kappa p(\tau)$ for every periodic orbit $\tau,$ \item[-] the entropy $h_f\in(0,\infty).$\end{itemize}
\end{lema}

Denote by $\holder^\alpha(X,V)$ the space of H\"older-continuous
$V-$valued maps with exponent $\alpha.$ For
$f\in\holder^\alpha(X,V)$ denote by $\|f\|_\infty :=\max |f|$ and
$$K_f=\sup \frac{\|f(p)-f(q)\|}{d(p,q)^\alpha},$$ one then defines
the norm of $f$ by $\|f\|_\alpha=\|f\|_\infty+K_f.$

The vector space $(\holder^\alpha(X,V),\|\ \|_\alpha)$ is a Banach
space and Liv\v sic's theorem implies that the vector space of
functions Liv\v sic-cohomologous to 0 is a closed subspace. Denote
by $\Livsic^\alpha(X,V)$ the quotient Banach space, and by $[\
]_L$ the projection.

Denote by $\Livsic_+^\alpha(X,\R)$ the subset of
$\Livsic^\alpha(X,\R)$ consisting of functions Liv\v
sic-cohomologous to a positive function.

\begin{lema}[{\cite[Lemma 2.13]{exponential}}]\label{lema:entropia} The entropy function $h:\Livsic^\alpha_+(X)\to\R_{>0},$ defined by $f\mapsto h_f,$ is analytic.
\end{lema}

Consider now a H\"older-continuous map $f:X\to V,$ and denote by
$\cone_f$ the closed cone of $V$ generated by the periods of $f$

$$\big\{\int_\tau f:\tau\textrm{ periodic}\big\}.$$
Assume its \emph{dual cone}, defined by $\cone_f^*=\{\varphi\in
V^*:\varphi|\cone_f\geq0\},$ is different from $\{0\}.$ The
\emph{entropy} of $\varphi\in\cone^*_f$ is defined by
$h^\varphi_f=h_{\varphi\circ f}.$ The following lemma is now direct using Lemma \ref{lema:positiva} (see also S. \cite[Lemma 3.2]{orbitalcounting}).

\begin{lema}\label{lema:interior} If there exists $\varphi\in\cone_f^*$ with finite entropy then it belongs to the interior of $\cone_f^*.$ If this is the case, any linear form $\varphi\in\cone^*_f$ has finite and positive entropy if and only if it belongs to the interior of $\cone_f^*.$
\end{lema}

We will assume from now on that there exists a linear form in $\cone_f^*$ with finite entropy.

In view of the last lemma, one considers the open subset of $\Livsic^\alpha(X,V)$ defined by $$\Livsic^\alpha_{+}(X,V)=\{[f]_L:\exists\varphi\in \cone_f^* \textrm{ with }h^\varphi_f\in(0,\infty)\}.$$

\begin{lema}\label{lema:cono} The map $\Livsic^\alpha_{+}(X,V)\to \{\textrm{compact subsets of }\P(V)\}$ defined by $$f\mapsto \P(\cone_f),$$ is continuous.
\end{lema}

\begin{proof}Recall that the space $\cal M^\phi$ of $\phi$-invariant probability measures is compact. Moreover, since $\phi$ is Anosov, periodic orbits viewed as invariant probability measures\footnote[2]{To a periodic orbit $\tau$ one associates the invariant probability measure ${\displaystyle r\mapsto\frac1{p(\tau)}\int_{\tau}r}$} are dense in $\cal M^\phi$ (c.f. Anosov's closing lemma, see Sigmund \cite{sigmund}). Consequently, the set $$\cal K_f=\{\int f dm:m\in\cal M^\phi\}$$ is compact and generates the cone $\cone_f.$ Moreover, $f\mapsto \cal K_f$ is continuous.

In order to show that its projectivisation is also continuous, we need to show that $0\notin\cal K_f,$ but since $\varphi(f)$ is Liv\v sic cohomologous to a positive function, there exists $k>0$ such that $\varphi(\int fdm)>k$ for all $m\in\cal M^\phi.$ This finishes the proof.
\end{proof}

Summarizing one obtains the following:

\begin{cor}\label{cor:entropiaanalitica} Consider $f_0\in\Livsic_+^\alpha(X,V)$ and $\varphi\in\inte\cone_{f_0}^*,$ then the entropy function defined by $f\mapsto h_f^\varphi$ is analytic on a neighborhood $U$ of $f_0$ such that $\varphi\in\inte\cone_f^*$ for all $f\in U.$
\end{cor}

We say that $f\in\Livsic_+^\alpha(X,V)$ is \emph{non-arithmetic on $V$} if the additive group generated by its periods is dense in $V.$ Consider the set

$$\D_f=\{\varphi\in V^*:P(-\varphi\circ f)\leq0\}.$$ It follows from the definition of pressure that $\D_f$ is convex,
and that if $\varphi\in \D_f$ then $t\varphi\in \D_f$ for all
$t\geq1.$

\begin{prop}[{\cite[Propositions 4.5 and 4.7]{exponential}}]\label{prop:convexo} The set $\D_f$ coincides with the set $\{\varphi\in\cone^*_f:h_f^\varphi\in(0,1]\},$ its boundary $\bord\D_f$ coincides with the set $$\{\varphi\in\cone^*_f:h^\varphi_f=1\},$$ and is a codimension $1$ closed analytic submanifold of $V.$ If moreover $f$ is non-arithmetic on $V,$ then $\D_f$ is strictly convex.
\end{prop}



\subsection{SRB measures and reparametrizations}\label{subsection:1u}

In this subsection we recall some classical results in the Sinai-Ruelle-Bowen theory and reinterpret them in the context of reparametrizations. It is common in the literature to state this type of results under a $\clase^2$-hypothesis. We shall explain how those results work in the $\clase^{1+\alpha}$-context.

%

Assume from now on that $X$ is a compact manifold and that the
flow $\phi$ is $\clase^1.$ We say that $\phi$ is \emph{Anosov} if
the tangent bundle of $X$ splits as a sum of three
$d\phi_t$-invariant bundles $$ TX=E^s\oplus E^0\oplus E^u,$$ and
there exist positive constants $C$ and $c$ such that: $E^0$ is the
direction of the flow and for every $t\geq0$ one has: for every
$v\in E^s$ $$\|d\phi_tv\|\leq Ce^{-ct}\|v\|,$$ and for every $v\in
E^u$ $$\|d\phi_{-t}v\|\leq Ce^{-ct}\|v\|.$$

If $\phi$ is an Anosov flow let $\lambda^u:X\to\R_+$ be the \emph{infinitesimal expansion rate} on the unstable direction, defined by

$$\lambda^u(x)= \left.\frac\partial{\partial t}\right|_{t=0} \frac{1}{\kappa} \int_0^\kappa \log\det( d_x\phi_{t+s}|E^u)ds$$ for some $\kappa>0.$

\begin{obs}\label{obs:expperiodo} Notice that by definition, if $\tau$ is a periodic orbit then $$\int_\tau \lambda^u=\log\det d_x\phi_{p(\tau)}|E^u,$$ for any $x\in\tau.$ Moreover, it is a direct consequence of Liv\v{s}ic's Theorem \ref{teo:livsic3} that the Liv\v{s}ic-cohomology class of $\lambda^u$ does not depend on $\kappa$, hence it will not appear in the notation.
\end{obs}

\begin{teo}[Sinai-Ruelle-Bowen \cite{bowenruelle}]\label{teo:srb} Let $\phi$ be a $\clase^{1+\alpha}$ Anosov flow on a compact manifold $X,$ then $P(-\lambda^u)=0.$
\end{teo}

This is statement is proved in Bowen-Ruelle \cite[Proposition 4.4]{bowenruelle} assuming $\phi$ is $\clase^2$. Let us now give some hints on why the proof carries on in the $\clase^{1+\alpha}$-setting. The $\clase^2$-hypothesis in \cite{bowenruelle} appears for three reasons:

\begin{itemize}
 \item[-] In order to guarantee that the function $x \mapsto E^u(x)$ is H\"older-continuous. This holds for $\clase^{1+\alpha}$ Anosov flows too (see for example Katok-Hasselblatt \cite[Proposition 19.1.6]{katokh}).
 \item[-] In order to show that $t \mapsto \log \det (d_x\phi_t|E^u)$ is $\clase^1$. By using our function $\lambda^u$ this is no longer necessary as long as we show that the volume lemma holds for $\lambda^u.$
 \item[-] To prove the \emph{volume lemma} (\cite[Lemma 4.2]{bowenruelle}) relating the function they define with the rate of decrease of the volume of Bowen balls. This can be proved in our context, for the function $\lambda^u$, by following the same scheme as \cite[Proposition 20.4.2]{katokh}.
\end{itemize}

Theorem \ref{teo:srb} together with Lemma \ref{lema:reparam} give immediately the following corollary.

\begin{cor}\label{cor:rep1} Let $\phi$ be a $\clase^{1+\alpha}$ Anosov flow, then the topological entropy of the reparametrization of $\phi$ by $\lambda^u$ is $1.$
\end{cor}

In section \ref{section:regularidad} we make use of the following well known result. Denote by $\lambda^s:X\to\R$ the infinitesimal expansion rate of the inverse flow $(\phi_{-t})_{t\in\R}.$

\begin{teo}[Sinai-Ruelle-Bowen \cite{bowenruelle}]\label{teo:volumen} Let $\phi$ be a $\clase^{1+\alpha}$ Anosov flow on a compact manifold $X,$ then $\phi$ preserves a measure in the class of Lebesgue if and only if $\lambda^u$ and $\lambda^s$ are Liv\v sic-cohomologous.
\end{teo}


\section{Projective Anosov representations}\label{section:convexA}

The main purpose of this section and Section
\ref{section:Generalanosov} is to extend several results from
\cite{quantitative} and \cite{exponential} to the Anosov
representations setting. We present here some general results from
\cite{presion} on \emph{projective Anosov representations}. These
representations are a basic tool to study general Anosov
representations (introduced by Labourie \cite{labourie}), as we
shall see in the next section. A more explanatory and detailed
exposition on this class of representations is Labourie
\cite{labourie}, Guichard-Wienhard \cite{olivieranna},
\cite{quantitative} and \cite{presion}.

Let $\G$ be a word hyperbolic group.

\begin{defi} A representation $\rho:\G\to\PGL(d,\R)$ has \emph{transverse maps} if there exist two continuous $\rho$-equivariant maps $(\xi,\xi^*):\bord\G\to\P(\R^d)\times \P((\R^d)^*)$ such that if $x\neq y$ then $\xi(y)\oplus\ker\xi^*(x)=\R^d.$
\end{defi}

In order to define the Anosov property for a representation with transverse maps, we need to recall the \emph{Gromov geodesic flow} of $\G.$ Gromov \cite{gromov} (see also Mineyev \cite{mineyev}) defines a proper cocompact action of $\G$ on $\bord^2\G\times\R,$ which commutes with the action of $\R$ by translation on the final factor. The action of $\G$ restricted to $\bord^2\G$ is the diagonal action.

There is a metric on $\bord^2\G\times\R$, well-defined up to
H\"older equivalence, so that $\G$ acts by isometries, every orbit
of the $\R$ action gives a quasi-isometric embedding and the
traslation flow on the $\R$-coordinate acts by bi-Lipschitz
homeomorphisms. This flow on  $\w\UG=\bord^2\G\times\R$ descends
to a flow $\phi$ on the quotient $\UG=\bord^2\G\times\R/\G.$ This
flow is called \emph{the geodesic flow} of $\G.$

If $\rho$ has transverse maps, the equivariant maps $(\xi,\xi^*)$ provide
two fiber bundles over $\w\UG,$ denoted by $\w\Xi$ and $\w\Theta$
respectively, whose fibers at $(x,y,t)\in\w\UG$ are respectively
$\w\Xi(x,y,t)=\xi(x)$ and
$\w\Theta(x,y,t)=\ker\xi^*(y).$  The diagonal
action of $\G$ on $\w\Xi$ and $\w\Theta$ is properly discontinuous
(because it is on $\w\UG$) and one obtains two vector bundles
$\Xi$ and $\Theta$ over $\UG.$

The geodesic flow of $\G$ on $\w\UG$ extends to $\w\Xi$ and
$\w\Theta$ by acting trivially on the fibers. This flow induces a
flow on the respective quotients. Denote by
$\psi=(\psi_t)_{t\in\R}$  the induced flow on the bundle
$\Xi^*\otimes\Theta.$

The representation $\rho$ is \emph{projective Anosov} if it has transverse maps and the flow
$\psi$ is contracting to the past, i.e. there exist $C,c>0$ such
that for all $w\in\Xi^*\otimes\Theta$ and $t>0$ one has
$$\|\psi_{-t}w\|\leq Ce^{-ct}\|w\|,$$ where $\|\ \|$ is a
Euclidean metric on the bundle $\Xi^*\otimes\Theta.$

For $g\in\PGL(d,\R),$ denote by $\lambda_1(g)$ the logarithm of
the spectral radius of some lift $\w g\in\GL(d,\R)$ of $g,$ with
$\det \w g\in\{-1,1\}.$ We say that $g$ is \emph{proximal} if the
generalized eigenspace of $\w g$ of eigenvalue with modulus
$e^{\lambda_1(g)}$ has dimension 1. Such eigenline, denoted by
$g_+,$ is an \emph{attractor} for $g$ on $\P(\R^d),$ and its
$g$-invariant complement $g_-$ (i.e. $\R^d=g_+\oplus g_-$) is its
\emph{repelling hyperplane}. The following lemma is standard (see
Guichard-Wienhard {\cite[Lemma 3.1]{olivieranna}}).

\begin{lema}\label{lema:proximal} Let $\rho$ be a projective Anosov representation, then for every non-torsion $\g\in\G,$ the element $\rho(\g)$ is proximal on $\P(\R^d),$ its attractive line is $\xi(\g_+)$ and its repelling hyperplane is $\ker\xi^*(\g_-).$
\end{lema}

The equivariant maps are unique, since they are continuous (in fact H\"older-continuous \cite[Lemma 2.5]{presion}) and uniquely defined on a dense set of $\bord\G.$

Denote by $\Lim_\rho=\xi(\bord \G)$ and by
$\Lim_\rho^*=\xi^*(\bord \G).$ If $\rho$ is irreducible, these are
the \emph{limit sets} (on $\P(\R^d)$ and $\P((\R^d)^*)$
respectively) of $\rho(\G),$ introduced by Guivarc'h
\cite{guivarch} and Benoist \cite{limite}. Denote by
$$\Lim_\rho^{(2)}=(\xi,\xi^*)(\bord^2\G)=\{(x,y)\in\Lim_\rho\times\Lim_\rho^*:
\R^d=\ker y\oplus x\}.$$

Consider the tautological bundle $\w \UG_\rho$ over $\Lim_\rho^{(2)},$ whose fiber at $(x,y)$ is defined by $${\sf{M}}_\rho(x,y)=\{(v,\varphi): v\in x, \varphi\in y\,\textrm{ and }\varphi(v)=1\}/(v,\varphi)\sim-(v,\varphi).$$ 
The bundle $\w \UG_\rho$ is equipped with a flow ${\w\phi}^\rho=({\w\phi_t}^\rho)$  defined by $${\w\phi}_t^\rho(x,y,(v,\varphi))=(x,y,(e^tv,e^{-t}\varphi)),$$ that commutes with the natural action of $\rho(\G).$ It is a consequence of the following theorem that the action of $\rho(\G)$ on $\w\UG_\rho$ is properly discontinuous and cocompact. The induced flow $\phi^\rho$ on the quotient $\UG_\rho=\rho(\G)\/ \w \UG_\rho$ is called \emph{the geodesic flow} of $\rho.$

\begin{teo}[{Bridgeman-Canary-Labourie-S. \cite[Section 4]{presion}}]\label{teo:flujo1} Let $\rho$ be a projective Anosov representation, then there exists a $\rho$-equivariant H\"older-continuous homeomorphism $E:\w \UG_\rho\to\w\UG,$ which is an orbit equivalence for the respective geodesic flows. The geodesic flow of $\rho$ is a transitive metric Anosov flow and its stable and unstable laminations are given by (the induced on the quotient of) $$\w W^s(x_0,y_0,(v_0,\varphi_0))=\{(x_0,y,(v_0,\varphi)):y\in\Lim_\rho^*-\{x_0\},\, \varphi\in y,\, \varphi(v_0)=1\}$$ and $$\w W^u(x_0,y_0,(v_0,\varphi_0))=\{(x,y_0,(v,\varphi_0)):x\in\Lim_\rho-\{y_0\},\, v\in x,\, \varphi_0(v)=1\}.$$
\end{teo}

Periodic orbits of $\phi^\rho$ are in bijective correspondence with conjugacy classes of primitive  elements of $\G$ (i.e. not a positive power of some other element in $\G$), namely, if $\g$ is such an element then its associated periodic orbit is the projection of $(\g_+,\g_-,(v,\varphi)),$ for (any) $\varphi\in\xi^*(\g_-)$ and $v\in\xi(\g_+). $

Since $\xi(\g_+)$ is the attracting line of $\rho(\g)$ (Lemma \ref{lema:proximal}), one obtains $$\g(\g_+,\g_-,(v,\varphi))=(\g_+,\g_-,( e^{\lambda_1(\rho\g)}v,e^{-\lambda_1(\rho\g)}\varphi)).$$ Consequently, the period of such periodic orbit is $\lambda_1(\rho\g).$

Hence, since the flows $\phi^\rho$ and $\phi$ are H\"older orbit equivalent, there exists a H\"older-continuous positive function
$f_\rho:\UG\to\R_+$ such that for every non-torsion $\g\in\G,$ one has $\int_\g f_\rho=\lambda_1(\rho\g).$ Such $f_\rho$ is unique up
to Liv\v sic-cohomology.

\begin{teo}[Bridgeman-Canary-Labourie-S. {\cite[Proposition 6.2]{presion}}]\label{teo:analytic0} Let $\{\rho_u:\G\to\PGL(d,\R)\}_{u\in D}$ be an analytic family\footnote[4]{We are assuming this family is far from the singular set of projective Anosov representations.} of projective Anosov representations. Then $u\mapsto [f_{\rho_u}]_L$ is analytic.
\end{teo}

The \emph{entropy} of $\rho$ is the topological entropy of the geodesic flow $\phi^\rho,$ and can be computed by $$h_\rho=\lim_{s\to\infty}\frac{\log\#\{[\g]\in[\G]\textrm{ non-torsion}:\lambda_1(\rho\g)\leq s\}}s.$$

\section{General Anosov
representations}\label{section:Generalanosov}

The concept of Anosov representation originated in Labourie \cite{labourie} and is further developed in Guichard-Wienhard \cite{olivieranna}.

Let $G$ be a real-algebraic semisimple Lie group. Let $K$ be a maximal compact subgroup of $G$ and $\tau$ the Cartan involution on $\frak g$ whose fixed point set is the Lie algebra of $K.$ Consider $\frak p=\{v\in\frak g: \tau v=-v\}$ and $\frak a$ a maximal abelian subspace contained in $\frak p.$

Let $\E$ be the set of roots of $\frak a$ on $\frak g,$ consider
$\frak a^+$ a closed Weyl chamber, $\E^+$ the set of positive
roots associated to $\frak a^+$ and $\Pi$ the set of simple roots
determined by $\E^+.$ To each subset $\t$ of $\Pi$ one associates
a pair of opposite parabolic subgroups $P_\t$ and
$\overline{P_\t}$ of $G,$ whose Lie algebras are, by
definition\footnote[2]{Note that we use the opposite convention
than Guichard-Wienhard \cite{olivieranna}, our $P_\t$ is their
$P_{\t^c}.$}, $$\frak p_\t=\frak g_0
\oplus\bigoplus_{\alpha\in\E^+}\frak g_\alpha\oplus
\bigoplus_{\alpha\in \<\Pi-\t\>}\frak g_{-\alpha}$$ and
$$\overline{\frak p_\t}=\frak g_0
\oplus\bigoplus_{\alpha\in\E^+}\frak g_{-\alpha}\oplus
\bigoplus_{\alpha\in \<\Pi-\t\>}\frak g_{\alpha}$$where $\<\t\>$
is the set of positive roots generated by $\t$ and $$\frak
g_\alpha=\{w\in\frak g:[v,w]=\alpha(v)w\ \forall v\in\frak a\}.$$

Let $W$ be the Weyl group of $\E$ and denote by $u_0:\frak a\to
\frak a$ the longest element in $W:$ i.e. $u_0$ is the unique
element in $W$ that sends $\frak a^+$ to $-\frak a^+.$  The
\emph{opposition involution} $\ii:\frak a\to\frak a$ is the
defined by $\ii=-u_0.$ Every parabolic subgroup is conjugated to a
unique $P_\t,$ in particular $\overline{P_\t}$ is conjugated to
$P_{\ii(\t)}$ where $$\ii(\t)=\{\alpha\circ\ii:\alpha\in\t\}.$$

Denote by $\scr F_\t=G/P_\t.$ The set $\scr F_{\ii(\t)}\times\scr F_\t$ possesses a unique open $G$-orbit, which we will denote by $\posgen_\t.$

\begin{ex}If $G=\PGL(d,\R)$ then $\frak a=\{(a_1,\ldots,a_d)\in\R^d:a_1+\cdots+a_d=0\},$ a Weyl chamber is $$\frak a^+=\{(a_1,\ldots,a_d)\in\frak a:a_1\geq\cdots\geq a_d\},$$ the set of positive roots associated to $\frak a^+$ is $\E^+=\{a\mapsto a_i-a_j:1\leq i<j\leq d\}$ and the simple roots are $\Pi=\{\sigma_i:i\in\{1,\ldots,d-1\}\}$ where $\sigma_i(a)=a_i-a_{i+1}.$ The opposition involution is $\ii(a)=(-a_d,\cdots,-a_1).$ 
\noindent
The parabolic group $P_\Pi$ is the stabilizer of a complete flag, and $\posgen_\Pi$ is the space of pairs of flags in general position, i.e. $(\{V_i\},\{W_i\}) \in \posgen_\Pi$ if $V_i \oplus W_{d-i}=\R^d$ for every $i$.
\end{ex}

Let $\G$ be a word hyperbolic group and consider a representation $\rho:\G\to G.$ Consider the trivial bundle $\w\UG\times\posgen_\t,$ and extend the geodesic flow of $\G$ to this bundle by acting trivially on the second coordinate. Passing to the quotient one obtains a flow $\phi$ on the bundle $\G\/ (\w\UG\times\posgen_\t)\to\UG.$

The representation $\rho$ is \emph{ $(P_\t,G)$-Anosov} if there exists a $\rho$-equivariant section $(\xi_\t,\xi_{\ii(\t)}):\w\UG\to\posgen_\t,$ invariant under the geodesic flow of $\G$ and such that its image is a hyperbolic set for $\phi$ whose stable distribution is the tangent space to $\{\cdot\}\times \scr F_{\ii(\t)}.$

Denote by $\homa_\t(\G,G)$ the space of $(P_\t,G)$-Anosov representations of $\G.$ Labourie \cite{labourie} and Guichard-Wienhard \cite{olivieranna} proved that this is an open subset of the space $\hom(\G,G).$

From the definitions one obtains that a representation is projective Anosov if and only if it is $({\sf{P}}_1,\PGL(d,\R))$-Anosov, where ${\sf{P}}_1$ is the stabilizer of a line in $\R^d.$ This follows from the following remark (see \cite[Proposition 2.11]{presion} for a proof).

\begin{obs}\label{obs:TP}Consider a decomposition $\R^d=\ell\oplus V,$ where $\ell$ is a line and $V$ a hyperplane, then the tangent space $T_\ell\P(\R^d)$ is canonically identified with $\hom(\ell,V).$
\end{obs}

Projective Anosov representations are useful to study general Anosov representations, as Theorem \ref{teo:oaconv} below shows. Let $\{\om_\alpha\}_{\alpha\in\Pi}$ be the set of fundamental weights of $\Pi.$

\begin{prop}[Tits \cite{tits}]\label{prop:titss} For each $\alpha\in\Pi$ there exists a finite dimensional proximal\footnote[2]{I.e. $\L_\alpha(G)$ contains a proximal matrix.} irreducible representation $\L_\alpha:G\to\PGL(V_\alpha),$ such that the highest weight $\chi_\alpha$ of $\L_\alpha$ is an integer multiple of the fundamental weight $\om_\alpha.$ All other weights of $\L_\alpha$ are of the form $$\chi_\alpha-\alpha-\sum_{\beta\in\Pi} n_\beta\beta$$ where $n_\beta\in\N.$
\end{prop}

In other words, if $g\in G$ then $\lambda_1(\L_\alpha(g))=k_\alpha\om_\alpha(\lambda(g)),$ where $\lambda:G\to\frak a^+$ is \emph{the Jordan projection} of $G.$

\begin{teo}[{Guichard-Wienhard \cite[Lemma 3.18+Theorem 4.10]{olivieranna}}]\label{teo:oaconv} Consider $\rho\in\homa_\t(\G,G),$ then for every $\alpha\in\t$ the composition $\L_\alpha\circ\rho:\G\to\PGL(V_\alpha)$ is projective Anosov.
\end{teo}

Let $$\frak a_\t=\bigcap_{\alpha\in\Pi-\t}\ker\alpha$$ be the Lie algebra of the center of the reductive group $P_\t\cap \vo{P_{\t}},$ where $\vo{P_\t}$ is the opposite parabolic group of $P_\t.$ Consider also $p_\t:\frak a\to\frak a_\t$ the only projection invariant under the group $W_\t=\{w\in W:w\textrm{ fixes pointwise }\frak a_\t\}.$ Note that, if $\alpha\in\t$ then $\om_\alpha=\om_\alpha\circ p_\t,$ (see for example Quint \cite[Lemme 2.2.3]{quint2}). Define $\lambda_\t:G\to\frak a_\t$ by $\lambda_\t=p_\t\circ\lambda.$ 

\begin{cor}\label{cor:ftheta} Consider $\rho\in\homa_\t(\G,G),$ then there exists a H\"older-continuous map $f_\rho^\t:\UG\to\frak a_\t,$ such that for every non-torsion conjugacy class $[\g]\in[\G]$ one has $$\int_{[\g]} f^\t_\rho=\lambda_\t(\rho\g).$$ Moreover, if $\{\rho_u\}_{u\in D}$ is an analytic family\footnote[4]{We are assuming this family is far from the singular set of $\homa_\t(\G,G).$} on $\homa_\t(\G,G),$ then $u\mapsto [f^\t_{\rho_u}]_L$ is analytic.
\end{cor}

\begin{proof}For\footnote[3]{The first statement is proved in \cite{exponential}, under the stronger hypothesis that $\rho(\G)$ is Zariski-dense.} each $\alpha\in\t$ the representation $\L_\alpha\circ\rho$ is projective Anosov (Theorem \ref{teo:oaconv}), hence Theorem \ref{teo:flujo1} guarantees the existence of a H\"older-continuous function $f_\rho^\alpha:\UG\to\R_+$ such that for all non-torsion $\g \in \G$ one has: 
$$\int_{[\g]} f^\alpha_\rho=\lambda_1(\L_\alpha\rho(\g))=k_\alpha\om_\alpha(\lambda(\rho\g)).$$ Note that, since $\alpha\in\t$ one has $\om_\alpha(\lambda(\rho\g))=\om_\alpha(\lambda_\t(\rho\g))$ (recall $\om_\alpha=\om_\alpha\circ p_\t$), and observe that the set of fundamental weights $\{\om_\alpha\}_{\alpha\in\t}$ is a basis of $\frak a_\t^*.$ Hence, there exists $f^\t_\rho:\UG\to\frak a_\t$ such that, for all $\alpha\in\t$ one has $$k_\alpha\om_\alpha(f^\t_\rho)=f^\alpha_\rho.$$ Theorem \ref{teo:analytic0} finishes the proof.
\end{proof}

\subsection{Limit cones}

Let $\grupo$ a discrete subgroup of $G.$ The \emph{limit cone} of $\grupo$ (introduced by Benoist \cite{limite}) is the closed cone generated by $\{\lambda(g):g\in\grupo\}$ and is denoted by $\cone_\grupo.$

\begin{prop}\label{prop:interior} Consider $\rho\in\homa_\t(\G,G).$ Then $\cone_{\rho(\G)}$ does not intersect the walls $\ker\alpha$ for every $\alpha\in\t\cup\ii(\t).$
\end{prop}

\begin{ex}\label{ex:contra}The proposition is optimal in the following sense: If $\rho:\pi_1\E\to\PSO(3,1)\subset\PSL(4,\R)$ is a quasi-Fuchsian representation then it is projective Anosov. Its limit cone is the Weil chamber of the Cartan algebra of $\PSO(3,1),$ which does not intersect the walls $\ker\sigma_1$ and $\ker\sigma_3$ but is contained in the wall $\ker\sigma_2.$
\end{ex}

\begin{proof} Assume first that $\rho:\G\to\PGL(d,\R)$ is projective Anosov. We have to show that its limit cone does not intersect the walls $\ker\sigma_1$ and $\ker\sigma_{d-1}.$

Consider a non-torsion element $\g\in\G.$ Recall that if $v\in\xi(\g_+)$ then $\rho(\g)v=\pm e^{\lambda_1(\rho\g)}v,$ and that $e^{\lambda_2(\rho\g)}$ is the spectral radius of $\rho(\g)|\ker\xi^*(\g_-).$ Consider a Euclidean metric $\{\|\ \|_p\}_{p\in\UG}$ on the bundle $\Xi^*\otimes\Theta.$ This metric lifts to a $\rho$-equivariant family of norms indexed on $\w\UG,$ still denoted by $\{\|\ \|_p\}_{p\in\w\UG}.$

Consider  $p=(\g_-,\g_+,t)\in\w\UG,$ $\varphi:\xi(\g_+)\to\R$ and $w\in\ker\xi^*(\g_-),$ then  $$\|\varphi\otimes w\|_{\phi_{-n|\g|}p}\leq Ce^{-n|\g|c}\|\varphi\otimes w\|_p.$$ Since $\phi_{-n|\g|}p=\g^{-n}p$ and the norms are equivariant, one has $\|\varphi\otimes w\|_{\phi_{-n|\g|}p}=\|\rho(\g^n) \varphi\otimes w\|_p,$ consequently $$e^{n(\lambda_2(\rho\g)-\lambda_1(\rho\g))}\|\varphi\otimes w\|_p \leq Ce^{-n|\g|c}\|\varphi\otimes w\|_p.$$ Hence $$\frac{\lambda_1(\rho\g)-\lambda_2(\rho\g)}{|\g|}>c,$$ for a $c>0$ independent of $\g.$ Finally, Theorem \ref{teo:flujo1} implies the existence of $M>m>0$ such that for every non-torsion $\g\in\G$ one has $$M>\frac{\lambda_1(\rho\g)}{|\g|}>m.$$ These two equations give $\cone_{\rho(\G)}\cap\ker\sigma_1=\{0\}.$ Since $\cone_\rho$ is $\ii$-invariant and $\sigma_{d-1}=\sigma_1\circ\ii,$ we obtain $\cone_\rho\cap\ker\sigma_{d-1}=\{0\}.$

Assume now that $\rho$ is $P_\t$-Anosov. Consider $\alpha\in\t$ and recall that $\L_\alpha\circ\rho$ is projective Anosov (Theorem \ref{teo:oaconv}). The proof finishes by applying the last paragraph to $\L_\alpha\circ\rho,$ and by recalling that there exists $k_\alpha\in\N$ such that for all $g\in G$ one has $$k_\alpha\alpha(\lambda(g))=\lambda_1(\L_\alpha g)-\lambda_2(\L_\alpha g).$$\end{proof}

If $\rho\in\homa_\t(\G,G)$ more information is given on the closed cone of $\frak a_\t$ generated by $\{\lambda_\t(\rho\g):\g\in\G\}.$ Denote this cone by $\cone^\t_\rho=\cone_{f^\t_\rho}$ (where $f^\t_\rho$ is given by Corollary \ref{cor:ftheta}), denote its dual cone by $\conodual=\{\varphi\in\frak a_\t^*:\varphi|\conodual\geq0\}.$ For $\varphi\in\conodual$ define its \emph{entropy} by $$h^\varphi_\rho=\lim_{s\to\infty}\frac{\log\#\{[\g]\in[\G]\textrm{ non-torsion}:\varphi(\lambda_\t(\rho\g))\leq s\}}s.$$ The following remark is direct from Lemma \ref{lema:interior}.

\begin{obs}\label{obs:interior2} A linear form $\varphi$ belongs to $\inte\conodual$ if and only if $h_\rho^\varphi\in(0,+\infty).$
\end{obs}

\begin{cor}\label{cor:conolimite} The function $\homa_\t(\G,G)\to\{\textrm{compact subsets of } \P(\frak a_\t)\}$ given by $\rho\mapsto\P(\cone_\rho^\t)$ is continuous. Consider $\rho_0\in\homa_\t(\G,G)$ and $\varphi\in\inte{\cone_{\rho_0}^\t}^*.$ Then the function $$\rho\mapsto h_\rho^\varphi$$ is analytic in a neighborhood $U$ of $\rho_0$ such that $\varphi\in\inte{\cone_\rho^\t}^*$ for every $\rho\in U.$

\end{cor}

\begin{proof} Follows from Corollary \ref{cor:ftheta}, Lemma \ref{lema:cono} and Corollary \ref{cor:entropiaanalitica}.
\end{proof}


We say that $\rho\in\homa_\t(\G,G)$ is \emph{non-arithmetic on} $\frak a_\t$ if the group generated by $\{\lambda_\t(\rho\g):\g\in\G\}$ is dense in $\frak a_\t.$ In the language of section \ref{section:1}, this is to say that the function $f_\rho^\t$ is non-arithmetic on $\frak a_\t.$

\begin{obs}\label{obs:non-arithmetic} Benoist's theorem \cite[Main theorem]{benoist2} asserts that if $\grupo$ is a Zariski-dense subgroup of $G,$ then the group generated by $\{\lambda(g):g\in\grupo\}$ is dense in $\frak a.$ Hence, if $\rho\in\homa_\t(\G,G)$ is Zariski-dense, then it is non-arithmetic on $\frak a_\t.$
\end{obs}


If $\rho\in\homa_\t(\G,G)$ denote by $\D_\rho^\t=\D_{f^\t_\rho}.$ The following is a direct consequence of Proposition \ref{prop:convexo}.

\begin{prop}\label{prop:convexoreps} Consider $\rho\in\homa_\t(\G,G),$ then the set  $$\bord\D_\rho^\t=\{\varphi\in\conodual:h_\rho^\varphi=1\},$$ is a codimension 1 closed analytic submanifold of $\frak a_\t^*.$ If moreover $\rho$ is non-arithmetic on $\frak a_\t,$ then the set $\D_\rho^\t=\{\varphi\in\conodual:h_\rho^\varphi\leq1\}$ is strictly convex.
\end{prop}

\section{The $i$-th eigenvalue}\label{section:i}

Let $\E$ be a closed orientable surface of genus $\geq2$ and
denote by $\G=\pi_1\E.$ Consider a $P_\Pi$-Anosov representation
$\rho:\G\to\PSL(d,\R)$ and denote by $\z:\bord\G\to\scr F$ its
equivariant map. We will say that $\z$ is a \emph{Frenet curve} if
for every decomposition $n=d_1+\cdots+d_k\leq d$ ($d_i\in\N$), and
$x_1,\ldots,x_k\in\bord\G$ pairwise distinct, one has that the
spaces $\z_{d_i}(x_i)$ are in direct sum, and moreover
$$\lim_{(x_i)\to x}\bigoplus_1^k\z_{d_i}(x_i)=\z_n(x),$$ where
$\z_i(x)$ is the $i$-dimensional space of the flag $\z(x).$

\begin{teo}[{Labourie \cite[Theorems 4.1 and 4.2]{labourie}}]\label{teo:hitchinFrenet} Consider $\rho\in\Hitchin(\E,d),$ then $\rho$ is $P_\Pi$-Anosov and $\z$ is a Frenet curve.
\end{teo}

There is a nice converse to this statement due to Guichard \cite{guichard}.

Denote by $\grassman_k(\R^d)$ the Grassmanian of $k$-dimensional
subspaces of $\R^d$. The Frenet condition implies
that if $d_1+d_2\leq d$ where $d_1,d_2\in\N,$ then the function
$\vo\z=\vo\z_{d_1,d_2}:(\bord\G)^2\to\grassman_{d_1+d_2}(\R^d)$
defined by

\begin{equation}\label{equation:zbarra}\vo\z(x,y)=\left\{\begin{array}{cc}\z_{d_1}(x)\oplus\z_{d_2}(y)\textrm{
if }x\neq y\\ \z_{d_1+d_2}(x)\textrm{ if
}x=y\end{array}\right.\end{equation} is (uniformly) continuous.

Labourie \cite{labourie} actually provides an even stronger transversality condition which he calls Property (H): given $x,y,z\in \bord\pi_1\E$ pairwise distinct then for every $i\in\{1,\ldots,d\}$ one has $$\z_{d-i+1}(y)\oplus(\z_{d-i+1}(z)\cap\z_i(x))\oplus \z_{i-2}(x)=\R^d.$$ By combining \cite[Proposition 8.2, Lemma 8.4, Lemma 9.1]{labourie} one obtains:

\begin{teo}[Labourie \cite{labourie}]\label{teo:H} The Frenet curve of a Hitchin representation verifies \emph{Property (H)}.
\end{teo}

For each $i\in\{1,\ldots, d\}$ consider the map
$\ell_i:\bord^{2}\G\to\P(\R^d)$ defined by
$$\ell_i(x,y)=\z_i(x)\cap\z_{d-i+1}(y).$$

\begin{figure}[ht]\begin{center}
\input{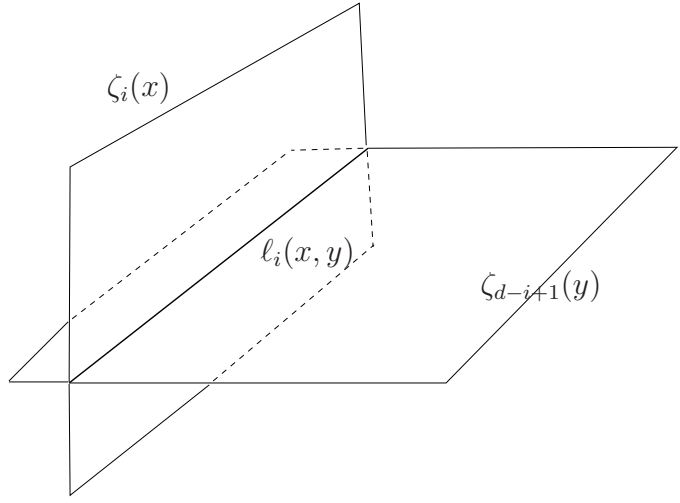}
\caption{\small{The $i$-th eigenvalue}}\label{figure-elli}
\end{center}\end{figure}

With this definition, Property (H) can be expressed as follows: For $x,z,t\in \bord \pi_1\E$ pairwise distinct one has: 

$$ \z_{d-i+1}(t) \oplus \ell_i(x,z) \oplus \z_{i-2}(x) = \R^d $$

\begin{obs} Note that each $\ell_i$ is H\"older-continuous and that for all non-torsion
$\g\in\G,$ the line $\ell_i(\g_+,\g_-)$ is the eigenline of
$\rho(\g)$ whose associated eigenvalue has modulus
$e^{\lambda_i(\rho\g)}.$ Observe also that $\ell_1(x,y)=\z_1(x)$
only depends on $x.$
\end{obs}

For $i\in\{2,\ldots,d-1\}$ let
$$E^u_i(x,y)=\hom(\ell_i(x,y),\ell_{i-1}(x,y))$$ and
$$E^s_i(x,y)=\hom(\ell_i(x,y),\ell_{i+1}(x,y)).$$
Notice that these bundles are H\"older-continuous on both variables. The purpose of this section is to prove the following proposition.

\begin{prop}\label{prop:ith-eigenvalue}
Consider $\rho\in\Hitchin(\E,d)$ and 
$2\leq i\leq d/2,$ then the space
$$\Lim^i_\rho=\{\ell_i(x,y):(x,y)\in\bord^2\G\}$$ is a
$\clase^{1+\alpha}$ submanifold of $\P(\R^d).$ The tangent space
to $\Lim^i_\rho$ at $\ell_i(x,y)$ is canonically identified with
$E^u_i(x,y)\oplus E^s_i(x,y).$
\end{prop}

This proposition implies the same statement for all
$i\in\{1,\ldots,d-1\}$ since $\ell_1(x,y)= \z_1(x)$ is $\clase^1$
by the Frenet property\footnote[2]{And indeed, the tangent space
can be expressed in terms of the function $\z_2$ and therefore it
is $\clase^{1+\alpha}$.}, and for $i>d/2$ one
has $\ell_{i}(x,y)=\ell_{d-i+1}(y,x)$.

\subsection{Proof of Proposition \ref{prop:ith-eigenvalue}}

Since $\rho$ is $P_\Pi$-Anosov, the map
$\ell_i:\bord^2\G\to\P(\R^d)$ is H\"older-continuous. Let us prove
that, except on special cases, it is injective. Indeed, notice
that if $i=1$ (resp. $i=d$) one has that $\ell_1(x,y)=\z_1(x)$
(resp. $\ell_d(x,y)=\z_1(y)$) and if $d=2k-1$ then $\ell_k$ is not
injective neither: $\ell_k(x,y)=\ell_k(y,x).$

\begin{lema}\label{lema:inyectivo} The map $\ell_i:\bord^2\G\to\P(\R^d)$ is injective for every $i\notin \{1, (d+1)/2,d\}$.
\end{lema}

\begin{proof} Assume first that $2\leq i <(d+1)/2.$ Thus, $2\leq i\leq d/2.$
Observe that, since $i+i\leq d,$ one has
$\z_i(x)\cap\z_i(y)=\{0\}$ for every $(x,y)\in\bord^2\G.$ Thus, if
$\ell_i(x,z)=\ell_i(y,t)$ then $x=y.$

Hence, we need to show that if $$\ell_i(x,z)=\ell_i(x,t)$$ then
$z=t.$ But if $x,z,t$ are pairwise distinct then Property (H) (Theorem \ref{teo:H}) implies $$\z_{d-i+1}(t)\oplus\ell_i(x,z)\oplus \z_{i-2}(x)=\R^d,$$ this contradicts the fact that $\ell_i(x,z)=\ell_i(x,t)\subset \z_{d-i+1}(t).$ 
Finally, if $i>(d+1)/2$ then $d-i+1<(d+1)/2$. The equality $\ell_i(x,y)=\ell_{d-i+1}(y,x)$ together with the last paragraph
gives injectivity. This finishes the proof.
\end{proof}

We need the following technical lemma.

\begin{lema}\label{lema:posgeneral}
Consider a $k$-dimensional vector subspace $W$ of $\R^d,$ and consider an incomplete flag $\{V_{d-k+i}:i\in\{0,\ldots,k\}\},$ such that $W\oplus V_{d-k}=\R^d.$  Then $\dim W\cap V_{d-k+i}=i.$
\end{lema}

\begin{proof} When $i=1$ the lemma follows easily. Assume now that the space $V'_i=W\cap V_{d-k+i}$ has dimension $i.$ Applying the  base step in the quotient space $\R^d/V'_i$ finishes the proof.
\end{proof}

We can now compute the 'partial derivatives' of $\ell_i.$ Define the
maps $e^u_i,e^s_i:\bord^{2}\G\to\grassman_2(\R^d)$ by
$$e^u_i(x,y)=\z_{i}(x)\cap\z_{d-i+2}(y)$$ and
$$e^s_i(x,y)=e^u_{d-i+1}(y,x)=\z_{i+1}(x)\cap\z_{d-i+1}(y).$$
Notice that injectivity implies that $\ell_i(x,y)+\ell_i(x,z)$ has
dimension $2$ (i.e. the sum is direct), we have the following:

\begin{lema}\label{lema:derivadas}
For $i\notin \{1, (d+1)/2,d\}$ and $x,y,z$ pairwise distinct, one has
$$\lim_{z\to y}\ell_i(x,z)\oplus\ell_i(x,y)=e^u_i(x,y),$$ and
$\lim_{z\to y}\ell_i(z,x)\oplus\ell_i(y,x)=e^s_i(y,x).$
\end{lema}

\begin{proof} The second statement follows from the first and the equalities $\ell_i(x,y)=\ell_{d-i+1}(y,x)$ and  $e^s_i(x,y)=e^u_{d-i+1}(y,x).$ We will focus hence on the first convergence. 

Since $\z_i(x)\cap\z_{d-i}(y)=\{0\},$ one has
$\z_{d-i+1}(y)=\z_{d-i}(y)\oplus \ell_i(x,y).$ Since $i \geq 2$
one has $(d-i+1)+1\leq d,$ and therefore the Frenet condition
implies
$$\z_1(z)\oplus\z_{d-i+1}(y)=\z_1(z)\oplus\z_{d-i}(y)\oplus
\ell_i(x,y).$$ Intersecting with $\z_i(x)$ one has $$(\z_1(z)\oplus\z_{d-i+1}(y))\cap\z_i(x)=(\z_1(z)\oplus\z_{d-i}(y)\oplus
\ell_i(x,y))\cap\z_i(x).$$ Since $\z$ is a Frenet curve Lemma \ref{lema:posgeneral} implies that the left hand side of the
equality has dimension 2 and also implies that $\dim
(\z_1(z)\oplus\z_{d-i}(y))\cap\z_i(x)=1.$ Since
$\ell_i(x,y)\in\z_i(x)$ we conclude that
\begin{equation}\label{eq:igualdad}(\z_1(z)\oplus\z_{d-i+1}(y))\cap\z_i(x)=([\z_1(z)\oplus\z_{d-i}(y)]\cap\z_i(x))\oplus\ell_i(x,y).\end{equation}

Given $\eps>0,$ consider $\delta>0$ from uniform continuity of
$\vo\z$ (equation (\ref{equation:zbarra})). If $d(z,y)\leq \delta$ then $\z_1(z)\oplus\z_{d-i+1}(y)$
is $\eps$-close to $\z_{d-i+2}(y),$ hence the left hand side of
equation (\ref{eq:igualdad}) is $\eps$-close to $e^u_i(x,y).$

Moreover, if $d(z,y)<\delta$ one has that
$\z_1(z)\oplus\z_{d-i}(y)$ is $\eps$-close to $\z_{d-i+1}(z).$
Thus $(\z_1(z)\oplus\z_{d-i}(y))\cap\z_i(x)$ is $\eps$-close to
$\ell_i(x,z).$ Furthermore $\ell_i(x,z)\cap\ell_i(x,y)=\{0\}$ since $z \neq y,$ hence the right hand side of equation
(\ref{eq:igualdad}) is $\eps$-close to
$\ell_i(x,z)\oplus\ell_i(x,y).$ Thus, equation (\ref{eq:igualdad})
implies that
$$d_{\grassman_2(\R^d)}(e^u_i(x,y),\ell_i(x,z)\oplus\ell(x,y))<2\eps.$$
\end{proof}

Using Lemmas \ref{lema:inyectivo}  and \ref{lema:derivadas} we can finish the proof of Proposition \ref{prop:ith-eigenvalue}

For $2\leq i\leq d-1,$ denote by
$\ell_i^*(x,y)=\z_{i-1}(x)\oplus\z_{d-i}(y)$ and note that
$\ell_i(x,y)\oplus\ell_i^*(x,y)=\R^d.$ Consider now the affine
chart of $\P(\R^d)$ defined by this decomposition, i.e. fix
$v\in\ell_i(x,y)$ and consider the map $\vartheta:\ell^*_i(x,y)\to
\P(\R^d)$ defined by $$w\mapsto \R(w+v).$$ This map identifies
$\ell^*_i(x,y)$ with $\P(\R^d-\P(\ell^*_i(x,y))).$

Denote by $w_i(a,b)\in\ell_i^*(x,y)$ the point defined by
$\vartheta(w_i(a,b))=\ell_i(a,b).$ This map may only be defined
near $(x,y),$ but this is not an issue. Observe that
$\vartheta^{-1}(\ell_i(x,z)\oplus\ell_i(x,y))$ is the straight
line defined by 0 and $w_i(x,z).$ The same holds for
$\vartheta^{-1}(\ell_i(z,y)\oplus\ell_i(x,y))$. Lemma
\ref{lema:derivadas} implies that the set
$\vartheta^{-1}\Lim^i_\rho$ has partial derivatives. Moreover,
these partial derivatives are H\"older-continuous since they can
be expressed in terms of the maps $\z_k.$

This implies that $\vartheta^{-1}\Lim^i_\rho$ is
$\clase^{1+\alpha}$ (near 0), and that its tangent space at $0$ is
$$\vartheta^{-1}(e^u_i(x,y))\oplus \vartheta^{-1}(e^s_i(x,y))=\ell_{i-1}(x,y)\oplus\ell_{i+1}(x,y).$$

We conclude that $\Lim^i_\rho$ is $\clase^{1+\alpha}$ and that its
tangent space at $\ell_i(x,y)$ is $E^u_i(x,y)\oplus E^s(x,y)$
(see Remark \ref{obs:TP}). This finishes the proof.
\begin{flushright}$\square$\end{flushright}

\section{Theorem C: The Anosov flow associated to $\ell_i$}\label{section:cociclosholder}

Let $\rho\in\Hitchin(\E,d),$ denote by $\G=\pi_1\E$ and consider the manifold $\Lim^i_\rho$ provided by Proposition \ref{prop:ith-eigenvalue}. Let $\w{\sf{F}}^i_\rho$ be the tautological line bundle over $\Lim^i_\rho$ whose fiber ${\sf{M}}^i_\rho(x,y)$ at $\ell_i(x,y)$ consists on
the elements of $\ell_i(x,y),$ i.e.
$${\sf{M}}^i_\rho(x,y)=\{v\in\ell_i(x,y)-\{0\}\}/v\sim-v.$$ The
fiber bundle $\w{\sf{F}}^i_\rho$ is equipped with the action of
$\rho(\G)$ and with a commuting $\R$-action, defined on each fiber
by $$\w\phi^i_t(v)=e^{-t}v.$$

Recall that $\frak a$ is the Cartan algebra of $\frak{sl}(d,\R)$ and that $\eps_i\in\frak a$ is defined by $\eps_i(a_1,\ldots,a_d)=a_i.$ The purpose of this section is to prove the following theorem.

\begin{teo}\label{teo:teoC} Assume $\cone_{\rho}\cap\ker \eps_i=\{0\},$ then there exists a $\rho$-equivariant H\"older-continuous homeomorphism $E:\w{\sf{F}^i_\rho}\to\w{\UG}$ that preserves the orbits of the respective flows.
\end{teo}

Consequently the action of $\rho(\G)$ on $\w{\sf{F}^i_\rho}$ is
properly discontinuous and cocompact and the quotient flow
$\phi^i$ on ${\sf{F}}^i_\rho=\rho(\G)\/\w{\sf{F}^i_\rho}$ is a
change of speed of the geodesic flow of $\G.$ Moreover one has the
following proposition.

\begin{prop}\label{prop:anosovflow} Assume $\cone_{\rho}\cap\ker \eps_i=\{0\},$ then $\phi^i$ is a $\clase^{1+\alpha}$ Anosov flow whose unstable distribution $E^u_i$ is given by (the induced on the quotient by)
$\hom(\ell_i(x,y),\ell_{i-1}(x,y)).$ Consequently the expansion rate
$\lambda^u:{\sf{F}}^i_\rho\to\R_+$ verifies that for every $\g \in \G$ one has
that: $$ \int_{[\g]} \lambda^u = \sigma_{i-1}(\lambda(\rho \g)).$$
\end{prop}

Lets prove Proposition \ref{prop:anosovflow} assuming Theorem \ref{teo:teoC}.

\begin{proof}
Since $\w{\sf{F}}^i_\rho$ is a $\clase^{1+\alpha}$ manifold and
the action of $\rho(\pi_1\E)$ on it is linear, we obtain that
${\sf{F}}^i_\rho=\rho(\pi_1\E)\/\w{\sf{F}}^i_\rho$ is
$\clase^{1+\alpha}$ and so is $\phi^i.$

Theorem \ref{teo:teoC} implies that $\phi^i$ is H\"older conjugate to a repara\-me\-tri\-zation of an Anosov flow (i.e. the geodesic flow of $\G$), hence it is metric Anosov with respect to the metric induced by the quotient: To prove this last assertion, the only thing to check is the existence of local (strong) stable and unstable manifolds since the uniform contraction and expansion follows from the fact that the reparametrizing function is positive. The existence of local (strong) stable and unstable
manifolds follows from classical graph transform arguments.

The differential $d\phi_t^i$ of $\phi_t^i$ preserves the
distribution $E^u_i$ induced on the quotient by
$\hom(\ell_i(x,y),\ell_{i+1}(x,y)).$ Along the periodic orbits, the local unstable
manifolds are tangent to $E^u_i$. Since the expansion of the local unstable manifolds is uniformly exponential, it follows that there exists $T$ such that for all $p$ in a periodic orbit one has $$\|d\phi_i^T|E_i^u(p)\|\geq2.$$ Since periodic orbits are dense and $E^u_i$ is continuous one concludes that 
$E^u_i$ is expanded uniformly in time. The symmetric argument gives
uniform contraction of $E^s_i.$

Finally, if $\g\in\G$ then recall that $\ell_i(\g_+,\g_-)$ is the eigenline of $\rho\g$ associated to the eigenvalue of modulus $\exp\lambda_i(\rho\g).$ Hence one has $$\g \cdot
(\ell_i(\g_+,\g_-), v)= (\ell_i(\g_+,\g_-), \rho \g (v))=\w{\phi^i}_{\lambda_i(\rho\g)}(\ell_i(\g_+,\g_-),v).$$

Thus, if one considers a $\G$-invariant Riemannian metric $\|\ \|$ on $\w{\sf{F}^i_\rho}$ and $\varphi \in
\hom(\ell_i (\g_+,\g_-),\ell_{i-1}(\g_+,\g_-))$ one has that $$\|d\w{\phi^i}_{\lambda_i(\rho\g)} (\varphi)\| = \|\g\cdot \varphi\|=\|\exp(\lambda_{i-1}(\rho\g)-\lambda_i(\rho\g))\varphi\|=\exp(\sigma_{i-1}(\lambda(\rho\g)))\|\varphi\|. $$

Hence Remark \ref{obs:expperiodo} implies that, for $x$ in the periodic orbit corresponding to $\g$ one has $$ \int_{[\g]}\lambda^u=\log \det (d_x\phi^i_{\lambda_i(\rho\g)}|E^u_i) = \sigma_{i-1}(\lambda(\rho\g)).  $$ This finishes the proof.

\end{proof}

Notice that Corollary \ref{cor:conolimite} implies that the map $\rho\mapsto\P(\cone_\rho)$ is continuous on $\Hitchin(\E,d)$ and hence $$U_i=\{\rho\in\Hitchin(\E,d):\cone_{\rho}\cap\ker \eps_i=\{0\}\}$$ 
is an open set. If $\rho_0$ is Fuchsian, then $$\cone_{\rho_0}=\frak a_{\PSL(2,\R)}^+=\{(d-1,d-3,\cdots,3-d,1-d)t:t\in\R_+\}.$$ Hence, if $i\in\{2,\ldots,d-1\}$ with $i\neq (d+1)/2$ then $\cone_{\rho_0}\cap\ker\eps_i=\{0\}.$ This is to say, the Fuchsian locus is contained in the open set $U=\bigcap_{i\neq (d+1)/2}U_i.$ One has the following corollary.

\begin{cor}[Theorem C]\label{cor:tuttipapel} If $\rho$ belongs to the neighborhood $U$ of the Fuchsian locus, then Proposition \ref{prop:anosovflow} holds for $\rho.$
\end{cor}

\subsection{H\"older cocycles} In this subsection we recall a basic tool of \cite{quantitative}. Consider a $\CAT(-1)$ space $X$ and denote by $\bord X$ its visual boundary. For a discrete subgroup $\G$ of $\isom X,$ denote by $\Lim_\G$ its limit set on $\bord X.$  Let $\w{\UG}$ denote the space of parametrized complete geodesics, $$\w\UG=\{\theta:(-\infty,\infty)\to X: \text{$\theta$ is a complete geodesic with $\theta_{-\infty},\theta_\infty\in\Lim_\G$}\}.$$

The group $\G$ naturally acts on $\w{\UG},$ and we denote by $\UG = \G\/\w{\UG}$ its quotient. We will say that $\G$ is \emph{convex cocompact} if the space $\UG$ is compact. If this is the case we will naturally identify $\Lim_\G$ with the Gromov boundary $\bord\G$ of $\G.$

We will now focus on cocycles for the action of $\G$ on $\bord^2\G=(\bord\G)^2-\{(x,x):x\in\bord\G\}.$ The main references for this subsection are Ledrappier \cite{ledrappier} and \cite[Section 5]{quantitative}. The usual setting is to consider cocycles on $\bord\G,$ however, it is convenient to use $\bord^2\G$ since our cocycles are naturally defined in this space.





\begin{defi}\label{defi:cociclo}A \emph{H\"older cocycle} is a function $c:\G\times\bord^2\G\to\R$ such that $$c(\g_0\g_1,x,y)=c(\g_0,\g_1(x,y))+c(\g_1,x,y)$$ for any $\g_0,\g_1\in\G$ and $(x,y)\in\bord^2\G,$ and where $c(\g,\cdot)$ is a H\"older map for every $\g\in\G$ (the same exponent is assumed for every $\g\in\G$).
\end{defi}

Given a H\"older cocycle $c$ and a non-torsion $\g\in\G,$ the \emph{period} of $\g$ for $c$ is defined by $$\l_c(\g)=c(\g,\g_+,\g_-),$$ where $\g_+$ is the attractive fixed point of $\g$ on $\bord\G,$ and $\g_-$ is the repelling one. The cocycle property implies that $\l_c(\g)$ only depends on the conjugacy class $[\g]\in[\G].$

Two H\"older cocycles $c,c'$ are \emph{cohomologous}, if there exists a H\"older-continuous function $U:\bord^2\G\to\R,$ such that for all $\g\in\G$ one has $$c(\g,x,y)-c'(\g,x,y)=U(\g x,\g y)-U(x,y).$$ 

\begin{teo}[Ledrappier \cite{ledrappier}]\label{teo:ledrappier} Let $c$ be a H\"older cocycle on $\bord^2\G,$ then there exists a H\"older-continuous function $f_c:\UG\to\R,$ such that for every non-torsion $[\g],$ one has $$\int_{[\g]}f_c=\l_c(\g).$$ 
\end{teo}

\begin{proof} This is a slight variation from Ledrappier's theorem, but the proof follows verbatim. Indeed, one can find an explicit formula for such $f_c$ as follows (Ledrappier \cite{ledrappier} page 105). Fix a point $o\in X$ and consider a $\clase^\infty$ function $F:\R\to\R$ with compact support such that $F(0)=1, F'(0)=F''(0)=0$ and $F(t)>1/2$ if $|t|\leq 2\sup\{d_X(p,\G\cdot o):p\in X\}.$

We can assume that $t\mapsto F(d_X(\theta(t),p))$ is
differentiable on $t$ for every $\theta\in\w{\UG}$ and $p\in X.$

Let $A:\w{\UG}\to\R$ be
\begin{equation}\label{eq:formula1}A(\theta)=\sum_{\g\in\G}F(d_X(\theta(0),\g
o))e^{-c(\g^{-1},\theta_{-\infty},\theta_\infty)}.\end{equation}
The function $f_c:\w{\UG} \to\R$ defined by
\begin{equation}\label{eq:formula}
f_c(\theta)=-\left.\frac{d}{dt}\right|_{t=0} \log
A(\w\phi_t\theta),\end{equation} where $\w\phi_t\theta\in\w{\UG}$ is the parametrized geodesic $s\mapsto \t(s+t),$ is $\G$-invariant and verifies
$\int_{[\g]} f_c=c(\g,\g_-,\g_+).$
\end{proof}

If $c$ is a H\"older cocycle with non-negative periods, one defines the \emph{entropy} of $c$ by $$h_c=\limsup_{t\to\infty}\frac 1t\log\#\{[\g]\in[\G]\textrm{ non torsion}:\l_c(\g)\leq t\}\in[0,\infty].$$


As in \cite{quantitative} one has the following reparametrizing theorem:

\begin{teo}[{\cite[Theorem 3.2]{quantitative}}]\label{teo:reparam}Let $c$ be a H\"older cocycle with non-negative periods and $h_c\in(0,\infty),$ then the action of $\G$ on $\bord^2\G\times\R$ defined by $$\g(x,y,t)=(\g x,\g y,t-c(\g,x,y))$$ is proper and cocompact. Moreover, the translation flow $\psi=(\psi_t)_{t\in\R}$ on the quotient $\G\/\bord^2\G\times\R$ is H\"older conjugated to a reparametrization of the geodesic flow of $\G.$ The topological entropy of $\psi$ is $h_c.$
\end{teo}

\begin{proof} The only difference between the actual statement of \cite[Theorem 3.2]{quantitative} is that the cocycle $c$ is defined on $\bord^2\G$ (as opposed to $\bord\G$), nevertheless the proof follows verbatim provided Ledrappier's Theorem \ref{teo:ledrappier}.
\end{proof}

\subsection{Proof of Theorem \ref{teo:teoC}}

Since $\cone_\rho\cap\ker\eps_i=\{0\}$ one has either
$\eps_i\in\inte\cone_\rho^*,$ or $-\eps_i\in\inte\cone_\rho^*.$ In
order to simplify notation assume $\eps_i\in\inte\cone_\rho^*.$
Remark \ref{obs:interior2} states that if this is the case then
$$h^{\eps_i}_\rho=\lim_{s\to\infty}\frac{\log\#\{[\g]\in[\pi_1\E]:\lambda_i(\rho\g)\leq
s\}}s\in(0,+\infty).$$

Consider a norm $\|\ \|$ on $\R^d.$ The H\"older cocycle $c:\pi_1\E\times\bord^2\pi_1\E\to\R,$ defined by $$c(\g,x,y)=\log\frac{\|\rho\g\cdot v\|}{\|v\|},$$ for any $v\in\ell_i(x,y),$ has periods $c(\g,\g_+,\g_-)=\lambda_i(\rho\g).$ Since $h_\rho^{\lambda_i}\in(0,\infty)$ the Reparametrizing Theorem \ref{teo:reparam} implies that the action of $\pi_1\E$ on $\bord^2\pi_1\E\times\R$ via $c,$ $$\g\cdot(x,y,t)=(\g x,\g y,t-c(\g,x,y))$$ is properly discontinuous and cocompact, moreover, the translation on the $\R$ coordinate is (conjugated to) a reparametrization of the geodesic flow of $\E$ (for a (any) hyperbolization on $\E$ fixed beforehand).

The proof of Theorem \ref{teo:teoC} is achieved by observing that the map
$\w{\sf{F}}^i_\rho\to\bord^2\pi_1\E\times\R$ defined by
$$(\ell_i(x,y),v)\mapsto (x,y,\log\|v\|)$$ is
$\pi_1\E$-equivariant for the cocycle $c$ (recall Lemma
\ref{lema:inyectivo}). This finishes the proof.


\section{Benoist Representations}\label{section:convex}

Let $\G$ be a hyperbolic group. A \emph{Benoist representation} is
a homomorphism $\rho:\G\to\PGL(n+1,\R)$ such that $\rho(\G)$
preserves an open convex set $\Om=\Om_\rho$ properly contained on an affine
chart, and such that the quotient $\rho(\G)\/\Om$ is compact.
Benoist
\cite{convexes1} has proved that under these conditions, 
the set $\Om$ is necessarily strictly convex and its boundary is a
$\clase^{1+\alpha}$ submanifold of $\P(\R^{n+1}).$

The geodesic flow $\phi=(\phi_t:\sf T^1(\rho(\G)\/\Om) \to \sf
T^1(\rho(\G)\/\Om))_{t\in\R}$ for the Hilbert metric on
$\rho(\G)\/\Omega$ is a $\clase^{1+\alpha}$ Anosov flow (Benoist
\cite{convexes1}). Denote by $\vo\varphi\in\frak a^*$ the
functional $\vo\varphi=(\eps_1-\eps_{n+1})/2.$ The topological
entropy of $\phi$ is
$$h_{\tex{top}}(\phi)=\lim_{s\to\infty}\frac{\log\#\{[\g]\in[\G]:\vo\varphi(\lambda(\rho\g))\leq
s\}}s.$$

Crampon \cite{crampon} has proved that $h_{\tex{top}}(\phi)\leq n-1,$ and equality only holds if $\Om$ is an ellipsoid, or equivalently, the Hilbert metric is Riemannian.

Benoist representation are projective Anosov representations, they are hence $P_\t$-Anosov where $\t=\{\sigma_1,\sigma_n\}\subset\Pi.$ Consider the vector space $\frak a_\t=\bigcap_{i=2}^{n-1}\ker\sigma_i.$ Its dual space $\frak a_\t^*\subset\frak a^*$ is spanned by the fundamental weights $\om_1(a)=\om_{\sigma_1}(a)=a_1$ and $$\om_n(a)=\om_{\sigma_n}(a)=\sum_1^n a_i=-a_{n+1}.$$ Denote by $\varphi^u,\varphi^s\in\frak a_\t^*$ the linear forms defined by $\varphi^u=n\om_1-\om_n$ and $\varphi^s=n\om_n-\om_1.$

Consider the expansion rate $\lambda^u:\sf
T^1(\rho(\G)\/\Om)\to\R_+$ of the geodesic flow $\phi.$ A standard
computation (for example Benoist \cite[Lemma 6.5]{convexes1})
shows that if $\g\in\G$ is primitive then
$$\int_{[\g]}\lambda^u=\sum_{i=2}^n(\lambda_1-\lambda_i)(\rho\g)=n\om_1(\lambda(\rho\g))-\om_n(\lambda(\rho\g))=\varphi^u(\lambda_\t(\rho\g)).$$

Corollary \ref{cor:rep1} and the last computation immediately imply the following.

\begin{cor} Let $\rho:\G\to\PGL(n+1,\R)$ be a Benoist representation, then $h^{\varphi^u}_\rho=h^{\varphi^s}_\rho=1.$
\end{cor}

Let $L$ be the positive cone generated by $\{\varphi^u ,\varphi^s\}$. Consider $\varphi \in \inte L$ and let $c(\varphi) \in \R_+$ be such that $c(\varphi)\varphi$ is a convex combination of $\varphi^u,\varphi^s$.

\begin{teo} For $\varphi \in\inte L$ one has that $h^{\varphi}_\rho \leq c(\varphi)$ and equality holds if and only if $\Om_\rho$ is an ellipsoid.
\end{teo}

\begin{proof} For a given $\rho$, we know that $\D_\rho^\t$ is a convex set whose boundary contains $\varphi^u$ and $\varphi^s$. This implies the inequality.

If equality holds then Proposition \ref{prop:convexoreps} implies
that $\rho$ is arithmetic in $\frak a_\t$, hence it is not Zariski-dense. Benoist's Theorem \cite[Theorem 3.6]{benoistclausura}
implies that $\Om$ is an ellipsoid.
\end{proof}

\begin{figure}[ht]\begin{center}
\input{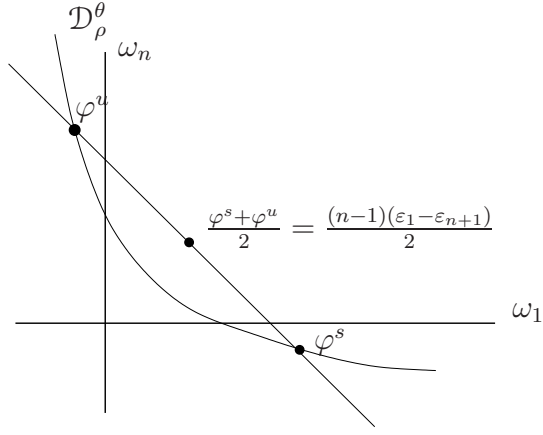}
\caption{\small{The set $\D^\t_\rho$ for a Benoist representation.}}\label{figure-Drhotheta}
\end{center}\end{figure}

Notice that $(n-1)\vo\varphi= \frac{\varphi^u + \varphi^s}{2},$ hence we obtain:

$$ h^{\vo\varphi}_\rho \leq n-1.$$

We end this section by observing that for $n+1=3$ one has $\frak a_\t =\frak a$ and $\D_\rho^\t= \D_\rho.$ Moreover, since $a_1+a_2+a_3=0$ one has $\varphi^u=\sigma_1$ and $\varphi^s=\sigma_2.$ Hence Theorem B is proved for $\Hitchin(\E,3)$.

\section{Theorem D: Regularity of the Frenet curve}\label{section:regularidad}

This section is devoted to the proof of Theorem D which states that if the Frenet equivariant curve $\z_1$ of a Hitchin representation $\rho$ is $\clase^\infty$, then $\rho$ is Fuchsian. 

We divide the proof in two steps: Proposition \ref{regular} states that if $\z_1$ is of class $\clase^\infty$ and $\rho$ belongs to a certain neighborhood of the Fuchsian locus then it is Fuchsian; the proof is completed with Proposition \ref{prop:local} which proves that if $\z_1$ is of class $\clase^\infty$ then necessarily $\rho$ belongs to this open set. 

In both cases, the proof uses the regularity to show that a certain Anosov flow preserves a volume form via a theorem of Ghys \cite{ghys}. Hence, Theorem \ref{teo:volumen} applies and one obtains relations between the eigenvalues of a given element. This idea is reminiscent of Benoist \cite[Section6.2]{convexes1}.

Recall that $U_i=\{\rho\in\Hitchin(\E,d): \cone_\rho\cap\ker \eps_i=\{0\}\}$ and $U = \bigcap_{i\neq(d+1)/2} U_i.$

\begin{prop}\label{regular}
Let $\rho$ be a Hitchin representation in the open set $U$. Assume moreover that $\z_1$ is of class
$\clase^\infty,$ then $\rho$ is Fuchsian.
\end{prop}

\begin{proof} Since $\z_1$ is $\clase^\infty,$ one has that
\begin{equation}\label{equation:derivadas} \z_i=\z_1\oplus\z_1'\oplus\cdots\oplus \z_1^{(i-1)}\end{equation} (Labourie
\cite{labourie}). The map $\z_i$ is thus $\clase^\infty$ and
therefore the manifold $\Lim^i_\rho$ is $\clase^\infty$.

Moreover, from the formula of the bundles $E^u$ and $E^s$ we
deduce that they are smooth bundles too. Applying a
result of Ghys \cite[Lemme 3.3]{ghys}\footnote[2]{The result of Ghys only requires $\clase^2$-regularity of the bundles (see \cite[Section 6]{ghys})
to provide a volume (contact) form invariant by the flow. This
allows to reduce the required regularity for the rigidity. Nevertheless, we do not know if this reduction is optimal.} we deduce that the flow $\phi^i$ preserves a volume form and hence $\lambda^u$ and $\lambda^s$ are Liv\v sic-cohomologous (Theorem \ref{teo:volumen}).

One concludes that for all $\g\in\pi_1\E$ and $i\in\{2,\ldots,d-1\}$ one has $\sigma_{i-1}(\lambda(\rho\g))=\sigma_i(\lambda(\rho\g))$. This implies that $\frak a_{G_\rho}=\frak a_{\tau_d(\PSL(2,\R))},$ hence $\rho$ is Fuchsian.
\end{proof}

\begin{prop}\label{prop:local} Let $\rho\in\Hitchin(\E,d)$ be such that $\z_1$ is of class $\clase^\infty.$ Then for all $i\in\{1,\ldots, d\}$ with $i\neq(d+1)/2$ one has $\cone_\rho\cap\ker\eps_i=\{0\}.$
\end{prop}

\begin{proof} Consider $2\leq i <(d+1)/2$ and consider the projective Anosov representation given by $\L^i\rho:\pi_1\E\to\PSL(\L^i\R^d).$ Its equivariant maps are given by $\xi=\L^i\z_i:\bord\pi_1\E\to \P(\L^i\R^d)$ and $\xi^*=\L^{d-i}\z_{d-i}:\bord\pi_1\E\to \P((\L^i\R^d)^*)$ (recall $\L^{d-i}\R^d$ is canonically isomorphic to $(\L^i\R^d)^*$).

Equation (\ref{equation:derivadas}) implies that $\xi(x)=\R( v_1\wedge\cdots\wedge v_i)$ where $v_j\in \z_1^{(j)}(x).$ Since $\z_1$ is of class $\clase^\infty$ we can compute $\xi'$ and one obtains (applying the product rule and observing that all terms but one have repetitions) $$\xi'(x)=\R( v_1\wedge\cdots\wedge v_{i-1}\wedge v_{i+1}).$$ Consequently, by Remark \ref{obs:TP} the tangent space $$T_{\xi(x)}\xi(\bord\pi_1\E)=\hom(\xi(x),\xi'(x)).$$ The geodesic flow of $\rho$ (recall Theorem \ref{teo:flujo1}) is a $\clase^\infty$ Anosov flow with $\clase^\infty$ distributions, namely $$E^u(x,y,(\varphi,v))=\hom(\xi(y),\xi'(y))\textrm{ and }E^s(x,y,(\varphi,v))=\hom(\xi^*(x),(\xi^*)'(x)).$$ A computation analogous to that of Proposition \ref{prop:anosovflow} gives $$\int_{[\g]}\lambda^u=\sigma_i(\lambda(\rho\g))\textrm{ and }\int_{[\g]}\lambda^s=-\sigma_{d-i}(\lambda(\rho\g))=\sigma_i\circ\ii(\lambda(\rho\g)).$$

Since the distributions are smooth, Ghys's result \cite[Lemme 3.3]{ghys} implies that the geodesic flow preserves a volume form and hence $\lambda^u$ and $\lambda^s$ are Liv\v sic-cohomologous, this implies that for all $\g\in\pi_1\E$ and $i\neq(d+1)/2$ one has $$\sigma_i(\lambda(\rho\g))=\sigma_i\circ\ii(\lambda(\rho\g)),$$ hence for all $j\in\{1,\ldots,d\}$ one has $\eps_j(\lambda(\rho\g))= -\eps_{d-j}(\lambda(\rho\g)).$

Since $\cone_\rho\subset \inte\frak a^+$ (Proposition \ref{prop:interior}) one deduces that $\cone_\rho\cap\ker\eps_i=\{0\}$ for all $i\neq(d+1)/2.$

\end{proof}

\bibliography{stage1}
\bibliographystyle{plain}

\author{$\ $\\ Rafael Potrie\\ CMAT Facultad de Ciencias\\ Universidad de la Rep\'ublica\\ Igu\'a 4225 Montevideo Uruguay\\
\texttt{rpotrie@cmat.edu.uy}}

\author{$\ $ \\
Andr\'es Sambarino\\
  IMJ-PRG (CNRS UMR 7586)\\ Universit\'e Pierre et Marie Curie (Paris VI) \\
  4 place Jussieu 75005 Paris France,\\
  \texttt{andres.sambarino@gmail.com}}

\end{document}